\documentclass[11pt, twoside, leqno]{article}

\usepackage{amssymb}
\usepackage{amsmath}
\usepackage{amsthm}
\usepackage{color}
\usepackage{mathrsfs}

\usepackage{indentfirst}

\usepackage{txfonts}

\allowdisplaybreaks

\def\red{\color{red}}

\def\rr{{\mathbb R}}
\def\rn{{\mathbb{R}^n}}

\def\zz{{\mathbb Z}}

\def\nn{{\mathbb N}}

\def\cf{{\mathcal F}}

\def\cs{{\mathcal S}}

\def\cy{{\mathcal Y}}

\def\fz{\infty }

\def\tz{\theta}

\def\lf{\left}
\def\r{\right}

\def\la{\langle}
\def\ra{\rangle}

\def\noz{\nonumber}

\def\loc{{\mathop\mathrm{\,loc\,}}}
\def\supp{\mathop\mathrm{\,supp\,}}

\def\fin{\mathop\mathrm{\,fin\,}}
\def\BMO{\mathop\mathrm{\,BMO\,}}
\def\bmo{\mathop\mathrm{\,bmo\,}}

\def\Xint#1{\mathchoice
{\XXint\displaystyle\textstyle{#1}}%
{\XXint\textstyle\scriptstyle{#1}}%
{\XXint\scriptstyle\scriptscriptstyle{#1}}%
{\XXint\scriptscriptstyle\scriptscriptstyle{#1}}%
\!\int}
\def\XXint#1#2#3{{\setbox0=\hbox{$#1{#2#3}{\int}$ }
\vcenter{\hbox{$#2#3$ }}\kern-.6\wd0}}

\def\dashint{\Xint-}

 \def\la{{\langle}}
 \def\ra{{\rangle}}
\def\({\left(}
\def \){ \right)}
 \def\supp{\operatorname{\,supp\,}}

 \def\sign{\operatorname{sign}}

\newtheorem{theorem}{Theorem}[section]
\newtheorem{lemma}[theorem]{Lemma}

\theoremstyle{definition}
\newtheorem{remark}[theorem]{Remark}
\newtheorem{definition}[theorem]{Definition}
\renewcommand{\appendix}{\par
   \setcounter{section}{0}%
   \setcounter{subsection}{0}%
   \setcounter{subsubsection}{0}%
   \gdef\thesection{\@Alph\c@section}%
   \gdef\thesubsection{\@Alph\c@section.\@arabic\c@subsection}%
   \gdef\theHsection{\@Alph\c@section.}%
   \gdef\theHsubsection{\@Alph\c@section.\@arabic\c@subsection}%
   \csname appendixmore\endcsname
 }

\numberwithin{equation}{section}

\begin{document}

\arraycolsep=1pt

\title{\bf\Large Real-Variable Characterizations of
Local Orlicz-Slice Hardy Spaces with Application to Bilinear Decompositions\footnotetext{\hspace{-0.35cm} 2010 {\it
Mathematics Subject Classification}. Primary 42B30;
Secondary 42B15, 42B35, 46E30, 42C40.
\endgraf {\it Key words and phrases.} Hardy space, Orlicz space,
atom, maximal function, bilinear decomposition.
\endgraf This project is supported
by the National Natural Science Foundation of China (Grant Nos.
11971058, 11761131002, 11671185 and 11871100).
}}
\author{Yangyang Zhang, Dachun Yang\,\footnote{Corresponding
author, E-mail: \texttt{dcyang@bnu.edu.cn}/{\red March 10, 2020}/Final version.}
\  and Wen Yuan}
\date{}
\maketitle

\vspace{-0.7cm}

\begin{center}
\begin{minipage}{13cm}
{\small {\bf Abstract}\quad
Recently, both the bilinear decompositions
$h^1(\mathbb R^n)\times \bmo(\mathbb{R}^n)\subset L^1(\mathbb R^n)+h_\ast^{\Phi}(\mathbb R^n)$
and $h^1(\mathbb R^n)\times \bmo(\mathbb{R}^n)\subset L^1(\mathbb R^n)+h^{\log}(\mathbb R^n)$
were established. In this article, the authors prove in some sense that the former is sharp, while the latter
is not. To this end, the authors first introduce the local Orlicz-slice Hardy space which contains
the variant $h_\ast^{\Phi}(\mathbb R^n)$ of the local Orlicz Hardy space
introduced by A. Bonami and J. Feuto
as a special case, and obtain its dual space by establishing its
characterizations via atoms, finite atoms and
various maximal functions, which are new even for $h_\ast^{\Phi}(\mathbb R^n)$.
The relationships $h_\ast^{\Phi}(\mathbb R^n)\subsetneqq h^{\log}(\mathbb R^n)$
is also clarified.}
\end{minipage}
\end{center}

\vspace{0.2cm}

\section{Introduction}\label{s1}

The bilinear decomposition of the product of Hardy spaces and their
dual spaces was originally studied by Bonami et al. \cite{BIJZ07}, which plays key roles
in improving the estimates of many nonlinear quantities
such as div-curl products, weak Jacobians (see, for instance,  \cite{BGK12,BFG,CLMS})
and commutators (see, for instance,  \cite{Ky13,LCFY}).
One of the most important result in this direction was made by Bonami et al. \cite{BGK12},
in which
Bonami et al.
proved the following
bilinear decomposition:
\begin{equation}\label{h123}
H^1(\rn)\times\BMO(\rn)\subset L^1(\rn)+H^{\rm log}(\rn),
\end{equation}
where $H^{\rm log}(\rn)$ was introduced in \cite{Ky14}, which denotes
the Musielak--Orlicz Hardy space related to the following Musielak--Orlicz function:
\begin{align}\label{2221}
\theta(x,\tau):=\frac{\tau}{\log(e+|x|)+\log(e+\tau)}
,\ \ \forall\,x\in\rn,\ \forall\,\tau\in [0,\fz).
\end{align}
Moreover, Bonami et al. in \cite{BGK12,bfgk} deduced that
\eqref{h123} is sharp in some sense and, in \cite{BK},
they proved that, in dimension
one, $H^{\rm log}(\rr)$ is indeed the smallest space, in the sense of the inclusion of sets, having the above property.
Recently, in \cite{bckly,blyy}, a bilinear decomposition theorem for multiplications of functions in
$H^p(\mathbb R^n)$
and its dual space $\mathfrak{C}_{\alpha}(\mathbb{R}^n)$
was established when $p\in(0,1)$ and $\alpha:=1/p-1$,
and the sharpness in some sense of this bilinear decomposition was also obtained therein.

For the local Hardy space,
Bonami et al. \cite{bf} established some \emph{linear} decomposition
of the product of the local Hardy space and its dual space.
Moreover, Bonami et al. \cite{bf} introduced the local Hardy-type space $h_*^\Phi(\rn)$, where
$\Phi(\tau):=\frac{\tau}{\log(e+\tau)}$ for any $\tau\in[0,\infty)$,
and proved the following \emph{linear} decomposition:
$$
h^1(\rn)\times\bmo(\rn)\subset L^1(\rn)+h_\ast^{\Phi}(\rn).
$$
Cao et al. \cite{cky1} obtained the \emph{bilinear} decomposition of product functions in the
local Hardy space $h^p(\rn)$ and its dual space for any $p\in(\frac{n}{n+1},1]$.
Recently,
a bilinear decomposition of the product of the
local Hardy space $h^p(\rn)$ and its dual space
with $p\in(0,1)$ was established in \cite{zyy} and
the sharpness of this bilinear decomposition was also obtained therein.
Observe that, for $p=1$, Cao et al. \cite{cky1} proved that
\begin{align}\label{cao1}
h^1(\rn)\times\bmo(\rn)\subset L^1(\rn)+h_\ast^{\Phi}(\rn)
\end{align}
and
\begin{align}\label{cao2}
h^1(\rn)\times\bmo(\rn)\subset L^1(\rn)+h^{\log}(\rn),
\end{align}
where $h^{\rm log}(\rn)$
denotes the local Hardy space of Musielak--Orlicz
type associated to the Musielak--Orlicz function
$\tz$ as in \eqref{2221} (see \cite{yy2}).
However, the sharpness of the bilinear decomposition
\eqref{cao1} and \eqref{cao2} is \emph{still unclear}.
Thus, it is a quite natural question to ask which one of \eqref{cao1} and \eqref{cao2}
is sharp and we give an affirmative answer to this question in this article by proving
that \eqref{cao1} is sharp, while \eqref{cao2} is not.
To this end, we need
first to give the dual space of the
local Hardy-type space $h_*^\Phi(\rn)$ in \cite{bf},
which forces us to establish real-variable characterizations of the local Hardy
type space $h_\ast^{\Phi}(\rn)$. For this purpose,
we introduce and study a new kind of local Hardy-type spaces,
the local Orlicz-slice Hardy space $(hE_\Phi^q)_t(\rn)$ (see Definition \ref{dh} below)
which contains the
local Hardy-type space $h_*^\Phi(\rn)$
as a special case. Compared with the sharpness
of the global case in \eqref{h123} on the Hardy
space $H^1(\rn)$, this sharpness of the local
case in \eqref{cao1} on the local Hardy space
$h^1(\rn)$ is a little bit surprising, which reveals the
essential difference between the homogeneous case and the
inhomogeneous case.

Let us give a brief review on the study of various variants of classical local Hardy spaces.
The classical local Hardy space $h^p(\rn)$ with $p\in (0, 1]$, introduced
by Goldberg in
\cite{g}, is known to be one of the most basic
working spaces on $\rn$ in harmonic analysis and partial differential equations, which plays key roles
in many branches of analysis; see, for instance, \cite{b81,g,r,t83,t92,t91} and their references.
In particular, pseudo-differential
operators are bounded on local Hardy spaces $h^p(\rn)$
with $p\in (0, 1]$, but they are not bounded on Hardy spaces $H^p(\rn)$
with $p\in (0, 1]$ (see, for instance, \cite{g,t83,t92}).
In recent decades, various variants of local Hardy spaces have been introduced and
developed; these variants include weighted local Hardy spaces (see, for instance, \cite{b81,t,r}),
weighted local Orlicz Hardy spaces (see, for instance, \cite{yy}) and
local Hardy-amalgam spaces (see, for instance, \cite{af2}).
Recently, in \cite{zyyw}, we introduced and studied a class of Orlicz-slice
spaces, $(E_\Phi^q)_t(\rn)$, which generalize the slice space studied by Auscher and Mourgoglou \cite{am}
as well as by Auscher and Prisuelos-Arribas \cite{ap}.
Based on these Orlicz-slice spaces, in \cite{zyyw}, we developed a real-variable theory of a
new kind of Hardy-type spaces, $(HE_\Phi^q)_t(\rn)$, which
contains the variant of the Orlicz Hardy space $H_*^\Phi(\rn)$ (see \cite{bf} of Bonami and Feuto)
and the Hardy-amalgam space (see \cite{af1} of Abl\'e and Feuto) as special cases.
However, no other real-variable theory of local Hardy-type spaces based on the Orlicz-slice space
$(E_\Phi^q)_t(\rn)$ is known so far.

The main target of this article is to determine
the sharpness of the bilinear decompositions \eqref{cao1} and \eqref{cao2}.
To achieve this,
we need to establish some real-variable characterizations of
the local Orlicz-slice Hardy space $(hE_\Phi^q)_t(\rn)$,
which is defined via the local maximal function
of Peetre type. This new scale of local Orlicz-slice Hardy spaces
contains the variant $h_*^\Phi(\rn)$ of the local Orlicz Hardy space
[in this case, $q=t=1$] as well as the local Hardy-amalgam space $\mathcal{H}^{(p,q)}_{\loc}(\rn)$
[in this case, $t=1$ and $\Phi(\tau):=\tau^p$ for any $\tau\in [0,\fz)$ with $p\in (0,\fz)$] of Abl\'e and Feuto \cite{af2}
as special cases. Their characterizations via the atom and
various maximal functions are also obtained.
We point out that,
compared with the atomic characterization of $\mathcal{H}_{\loc}^{(p,q)}(\rn)$
obtained in \cite{af2}, the atomic characterization of $(hE_\Phi^q)_t(\rn)$ obtained
in this article holds true on wider ranges of $p$ and $q$ even when $(hE_\Phi^q)_t(\rn)$ is reduced
to $\mathcal{H}^{(p,q)}_{\loc}(\rn)$, which
hence improves the corresponding result in \cite{af2}.
We then establish finite atomic characterizations
of $(hE_\Phi^q)_t(\rn)$, which further induces a description of their dual spaces.
These characterizations are new even for local slice Hardy spaces
based on the slice spaces in \cite{ap}.
Thus, the results obtained in this article essentially
generalize the corresponding real-variable theories of the
local Hardy-amalgam space in \cite{af2}
as well as the local Hardy-type space $h_*^\Phi(\rn)$ in \cite{bf}.
Then, as an application, we obtain the
characterization of the pointwise multipliers on $\bmo(\mathbb{R}^n)$
[see Theorem \ref{cheng}(iii) below]
and the sharpness of the bilinear decomposition \eqref{cao1} (see Remark \ref{mm} below).
Moreover, by showing that
$h_\ast^{\Phi}(\rn)\subsetneqq h^{\log}(\rn)$
[see Theorem \ref{cheng}(v) below],
we conclude that the bilinear decomposition \eqref{cao2} is not sharp (see Remark \ref{mm} below for the details).

To be precise, this article is organized as follows.

Section \ref{s3} is devoted to establishing real-variable characterizations of
the local Hardy-type space $(hE_\Phi^q)_t(\rn)$.
We first recall the notion
of Orlicz-slice spaces $(E_\Phi^q)_t(\rn)$ and some basic properties of
$(E_\Phi^q)_t(\rn)$.
Then, based on the Orlicz-slice space $(E_\Phi^q)_t(\rn)$,
we introduce
the local Orlicz-slice Hardy space,
$(hE_\Phi^q)_t(\rn)$
(see Definition \ref{dh} below).
Applying the real-variable theory of local Hardy spaces related to ball
quasi-Banach function spaces  developed in \cite{ykds},
we directly obtain their characterizations via various maximal functions.
Moreover, we further establish the atomic characterization (see Theorem \ref{atom ch} below)
and their finite atomic characterization (see Theorem \ref{finite} below).
Applying the last mentioned both characterizations of
$(hE_\Phi^q)_t(\rn)$,
we prove that the dual space of $(hE_\Phi^q)_t(\rn)$
is certain local Campanato spaces related to
the local Orlicz-slice space (see Theorem \ref{dual2} below).

In Section \ref{s5}, by means of the dual space of $h_*^\Phi(\rn)$,
we
characterize the pointwise multipliers of the local $\BMO$ space $\bmo(\mathbb{R}^n)$
[see Theorem \ref{cheng}(iii) below]
and the sharpness of the bilinear decomposition \eqref{cao1} (see Remark \ref{mm} below).
Moreover,
we show that the bilinear decomposition \eqref{cao2} is not sharp (see Remark \ref{mm} below).

Finally, we make some convention on notation.
For any $x\in\rn$ and $r\in(0,\infty)$, let $B(x,r):=\{y\in\rn:|x-y|<r\}$ and
$L^1_\loc(\rn)$ the set of all locally integrable functions on $\rn$.
For any set $E$, we use $\mathbf1_{E}$ to denote its \emph{characteristic function}
and $E^{\complement}:=\rn\setminus E$.
We also use $\vec{0}_n$ to denote the \emph{origin}
of $\rn$.
Let $\mathcal{S}(\rn)$ denote the collection of all
\emph{Schwartz functions} on $\rn$, equipped
with the classical well-known topology
determined by a sequence of norms, and $\mathcal{S}'(\rn)$ its \emph{topological dual}, namely,
the set of all bounded linear functionals on $\mathcal{S}(\rn)$
equipped with the weak-$\ast$ topology.
Let $\mathbb{N}:=\{1,\,2,...\}$ and $\mathbb{Z}_+:=\mathbb{N}\bigcup\{0\}$.
Throughout this article, we use ${\mathcal Q}$ to denote the set
of all cubes in $\rn$ having their edges parallel to the coordinate axes,
which are not necessary to be closed or open.
For any cube $Q:=Q(x_Q,\ell(Q))\in{\mathcal Q}$,
we denote by $x_Q$ its \emph{center} and by $\ell(Q)$ its \emph{side length},
and, for any $\alpha\in(0,\infty)$, let $\alpha Q:=Q(x_Q,\alpha \ell(Q))$.
For any $\varphi\in\mathcal{S}(\rn)$ and $t\in(0,\infty)$,
let $\varphi_t(\cdot):=t^{-n}\varphi(t^{-1}\cdot)$. For any $s\in\mathbb{R}$, we denote by $\lfloor s\rfloor$ the \emph{largest integer not greater than} $s$.
We always use $C$ to denote a \emph{positive constant}, which is independent of the main parameter,
but it may vary from line to line.
Moreover, we use $C_{(\gamma,\ \beta,\ \ldots)}$ to denote a positive constant depending on the indicated
parameters $\gamma,\ \beta,\ \ldots$. If, for any real functions $f$ and $g$, $f\leq Cg$, we then write
$f\lesssim g$ and, if $f\lesssim g\lesssim f$, we then write $f\sim g$. For any $\alpha:=(\alpha_1,\ \ldots,\ \alpha_n)\in\zz_+^n$ and $x:=(x_1,\ldots,x_n)\in\rn$, define
$|\alpha|:=\alpha_1+ \cdots+\alpha_n$,
$\partial^\alpha:=\partial_{x_1}^{\alpha_1}\cdots\partial_{x_n}^{\alpha_n}$ with
$\partial_{x_j}:=\frac{\partial}{\partial x_j}$ for any $j\in\{1,\ldots,n\}$ and
$x^\alpha:=x_1^{\alpha_1}\cdots x_n^{\alpha_n}$. For any $r\in [1,\fz]$, we use
$r'$ to denote its \emph{conjugate index}, namely, $1/r+1/r'=1$.

\section{Local Orlicz-slice Hardy spaces \label{s3}}

In this section, we introduce the local Orlicz-slice Hardy space via the local maximal function of Peetre
type. We then
establish several real-variable characterizations of the local Orlicz-slice Hardy space,
respectively, in terms of
various local maximal functions, atoms and finite atoms. As an application,
we determine the dual space of the
local Orlicz-slice Hardy space $(hE_\Phi^q)_t(\rn)$ with $\max\{p_{\Phi}^+,\ q\}\in(0,1]$.

\subsection{Maximal function characterizations
of local Orlicz-slice Hardy spaces\label{s3.1}}

In this subsection, we introduce the local
Orlicz-slice Hardy space $(hE_\Phi^q)_t(\rn)$ and
establish its various maximal function characterizations.
We begin with recalling the notions of both Orlicz functions
and Orlicz spaces (see, for instance, \cite{mmz}).

\begin{definition}\label{d1.1}
A function $\Phi:\ [0,\infty)\ \to\ [0,\infty)$ is called an \emph{Orlicz function} if it is
non-decreasing and satisfies $\Phi(0)= 0$, $\Phi(\tau)>0$ whenever $\tau\in(0,\infty)$ and $\lim_{\tau\to\infty}\Phi(\tau)=\infty$.
\end{definition}

An Orlicz function $\Phi$  is said to be
of \emph{lower} (resp., \emph{upper}) \emph{type} $p$ with
$p\in(-\infty,\infty)$ if
there exists a positive constant $C_{(p)}$, depending on $p$, such that, for any $\tau\in[0,\infty)$
and $s\in(0,1)$ [resp., $s\in [1,\infty)$],
\begin{equation*}
\Phi(s\tau)\le C_{(p)}s^p \Phi(\tau).
\end{equation*}
A function $\Phi:\ [0,\infty)\ \to\ [0,\infty)$ is said to be of
 \emph{positive lower} $p$ (resp., \emph{upper}) \emph{type} if it is of lower
 (resp., upper) type $p$ for some $p\in(0,\infty)$.

\begin{definition}\label{d1.2}
Let $\Phi$ be an Orlicz function with positive lower type $p_{\Phi}^-$ and positive upper type $p_{\Phi}^+$.
The \emph{Orlicz space $L^\Phi(\rn)$} is defined
to be the set of all measurable functions $f$ such that
 $$\|f\|_{L^\Phi(\rn)}:=\inf\lf\{\lambda\in(0,\infty):\ \int_{\rn}\Phi\lf(\frac{|f(x)|}{\lambda}\r)\,dx\le1\r\}<\infty.$$
\end{definition}

We now introduce the following Orlicz-slice space, which is a
generalization of the slice space introduced in \cite{ap} (see also \cite{am})
and was first introduced in \cite[Definition 2.8]{zyyw}.

\begin{definition}\label{d2}
Let $t,\ q\in(0,\infty)$ and $\Phi$ be an Orlicz function with positive lower type $p_{\Phi}^-$ and positive upper type $p_{\Phi}^+$. The \emph{Orlicz-slice space} $(E_\Phi^q)_t(\rn)$ is defined to be the set of all measurable functions $f$
such that $$\|f\|_{(E_\Phi^q)_t(\rn)}
:=\lf\{\int_{\rn}\lf[\frac{\|f\mathbf1_{B(x,t)}\|_{L^\Phi(\rn)}}
{\|\mathbf1_{B(x,t)}\|_{L^\Phi(\rn)}}\r]^q\,dx\r\}^{\frac{1}{q}}<\infty.$$
\end{definition}

\begin{remark}
\begin{enumerate}
\item[\rm{(i)}] The Orlicz-slice space is a quasi-Banach space (see \cite[Remark 2.10(i)]{zyyw}.
Let $q\in[1,\infty)$ and $\Phi$ be an Orlicz function with lower type
$p_{\Phi}^-\in[1,\infty)$ and positive upper type $p_{\Phi}^+$.
In this case, we may \emph{always assume} that $(E_\Phi^q)_t(\rn)$
is a Banach space (see \cite[Remark 3.17]{zyyw}.

\item[\rm{(ii)}]Let $q\in(0,\infty)$. If $t=1$ and $\Phi(\tau):=\tau^p$
for any $\tau\in [0,\fz)$ with $p\in (0,\fz)$,
then $(E_\Phi^q)_t(\rn)$ coincides with the Wiener amalgam spaces $(L^p,\ell^q)(\rn)$ in \cite{af1}.
If $\Phi(\tau):=\tau^r$ for any $\tau\in[0,\fz)$ with $r\in(0,\fz)$,
then $(E_\Phi^q)_t(\rn)$ and $(E_r^q)_t(\rn)$ from \cite{am,ap} coincide with equivalent
quasi-norms.
If $\Phi(\tau):=\tau^q$ for any $\tau\in[0,\fz)$,
then $(E_\Phi^q)_t(\rn)$ and $L^q(\rn)$ coincide with the same quasi-norms
(see \cite[Proposition 2.11]{zyyw}).
\end{enumerate}
\end{remark}

Recall that the \emph{centered Hardy-Littlewood maximal operator} $\mathcal{M}$ is defined by setting, for
any locally integrable function $f$ and any $x\in\rn$,
$$
\mathcal{M}(f)(x):=\sup_{r>0}\dashint_{B(x,r)}|f(y)|\,dy.
$$

The following lemma is a Fefferman--Stein type inequality for Orlicz-slice spaces, which is just
\cite[Theorem 2.20]{zyyw}.
\begin{lemma}\label{main}
Let $t\in(0,\infty)$, $q,\ r\in(1,\infty)$ and $\Phi$ be an Orlicz function with positive lower type $p_{\Phi}^-\in(1,\infty)$ and positive upper type $p_{\Phi}^+$.
Then there exists a positive constant $C$,
independent of $t$,
such that, for any $\{f_j\}_{j\in\mathbb{Z}}\subset (E_\Phi^q)_t(\rn)$,
$$
\lf\|\lf\{\sum _{j\in\mathbb{Z}}\lf[\mathcal{M}(f_j)\r]^r\r\}^{\frac{1}{r}}\r\|_{(E_\Phi^q)_t(\rn)}
\leq C\lf\|\lf\{\sum _{j\in\mathbb{Z}}|f_j|^r\r\}^{\frac{1}{r}}\r\|_{(E_\Phi^q)_t(\rn)}.
$$
\end{lemma}

In what follows,  for any $N\in\nn$, let
\begin{equation}\label{EqcfN}
\cf_N(\rn):=\lf\{\varphi\in\cs(\rn):\ \sum_{\beta\in\zz_+^n,|\beta|\le N}
\sup_{x\in\rn}\lf[\lf(1+|x|\r)^{N+n}\lf|\partial^\beta\varphi(x)\r|\r]\le1\r\},
\end{equation}
here and thereafter, for any $\beta:=(\beta_1,\ldots,\beta_n)\in\zz_+^n$
and $x:=(x_1,\ldots,x_n)\in\rn$,
$|\beta|:=\beta_1+\cdots+\beta_n$,
$\partial^\beta:=(\frac{\partial}{\partial x_1})^{\beta_1}
\cdots(\frac{\partial}{\partial x_n})^{\beta_n}$
and $x^\beta:=x_1^{\beta_1}\cdots x_n^{\beta_n}$.
Now we introduce the following notions of the local radial functions
and the local non-tangential maximal functions.

\begin{definition}\label{jdhs}
Let $a,\,b\in(0,\infty)$, $N\in\mathbb{N}$, $\varphi\in\mathcal{S}(\rn)$ and $f\in\mathcal{S}'(\rn)$.
\begin{enumerate}

\item[{\rm(i)}] The \emph{local radial maximal function $m(f,\varphi)$} of $f$ associated to $\varphi$
is defined by setting, for any $x\in\rn$,
$$
m(f,\varphi)(x):=\sup_{s\in(0,1)}|f\ast\varphi_s(x)|,
$$
where, for any $s\in(0,\infty)$ and $x\in\mathbb R^n,\varphi_s(x):=s^{-n}\varphi(x/s)$.

\item[{\rm(ii)}] The \emph{local grand maximal function} $m_N(f)$ is defined by setting, for any $x\in\rn$,
$$m_N(f)(x):=\sup\lf\{|\varphi_\tau  \ast f(y)|:\ \tau\in(0,1),\,|x-y|<\tau,\,
 \varphi\in\mathcal{F}_N(\rn)\r\}.$$

\item[{\rm(iii)}]  The \emph{local non-tangential maximal function} $m_a^*(f,\varphi)$, with aperture $a\in(0,\infty)$,
 is defined by setting, for any $x\in\rn$,
$$m_a^*(f,\varphi)(x):=\sup_{\tau\in(0,1)}\lf\{\sup_{y\in\rn,|y-x|<a\tau}|(\varphi_\tau \ast f)(y)|\r\}.$$

\item[{\rm(iv)}] The \emph{local maximal function $m_b^{**}(f,\varphi)$ of Peetre type} is defined by setting, for any $x\in\rn$,
\begin{equation}\label{51}
m_b^{**}(f,\varphi)(x):=\sup_{(y,\tau)\in\mathbb{R}^{n}\times(0,1)}\frac{|(\varphi_\tau \ast
 f)(x-y)|}{(1+\tau^{-1}|y|)^b}.
\end{equation}

\item[{\rm(v)}] The \emph{local grand maximal function $m(f,\varphi)$ of Peetre type} is defined by setting, for any $x\in\rn$,
$$m_{b,\,N}^{**}(f)(x):=\sup_{\psi\in\mathcal{F}_N(\rn)}\lf\{\sup_{(y,\tau)
\in\mathbb{R}^{n}\times(0,1)}\frac{|(\psi_\tau \ast f)(x-y)|}{(1+\tau^{-1}|y|)^b}\r\}.$$
\end{enumerate}
\end{definition}

\begin{definition}\label{dh}
Let $t$, $q\in(0,\infty) $ and $\Phi$ be an Orlicz function with positive lower type $p_{\Phi}^-$ and positive upper type $p_{\Phi}^+$.
Then the \emph{local Orlicz-slice Hardy space} $(hE_\Phi^q)_t(\rn)$ is defined by setting
$$
(hE_\Phi^q)_t(\rn):=\lf\{f\in\mathcal{S}'(\rn):\ \|f\|_{(hE_\Phi^q)_t(\rn)}:=
\|m_b^{**}(f,\varphi)\|_{(E_\Phi^q)_t(\rn)}<\infty \r\},
$$
where $\varphi\in\mathcal{S}(\rn)$ satisfies $\int_{\rn}\varphi(x)\,dx\neq0$ and $m_b^{**}(f,\varphi)$ is as in (\ref{51}) with $b$
sufficiently large (see Remark \ref{214x} below).
\end{definition}

\begin{remark}\label{ReHpq}
\begin{enumerate}
\item[(i)] If $t=1$, $\Phi(\tau):=\tau^p$
for any $\tau\in [0,\fz)$ with $p\in (0,\fz)$ and $q=p$,
then $(hE_\Phi^q)_t(\rn)$ coincides with the local Hardy space $h^{p}(\rn)$.
\item[(ii)] If $t=q=1$ and $\Phi(\tau):=\frac{\tau}{\log(e+\tau)}$ for any $\tau\in [0,\fz)$,
by \cite[Proposition 2.12]{zyyw}, we know that $(hE_\Phi^q)_t(\rn)$ coincides with
 the variant $h_*^\Phi(\rn)$ of the Orlicz Hardy space of Bonami and Feuto in \cite{bf}.
\end{enumerate}
\end{remark}

To obtain various maximal function characterizations of
the local Hardy-type space $(hE_\Phi^q)_t(\rn)$, we
first recall the notion of ball quasi-Banach
function spaces defined in \cite[Definition 2.1]{ykds}. In what follows,
denote by the \emph{symbol} $\mathbb{M}(\rn)$ the set of
all measurable functions on $\rn$.

\begin{definition}\label{ball}
A quasi-Banach space $X\subset\mathbb{M}(\rn)$ is called a \emph{ball quasi-Banach function space} on $\rn$
if it satisfies
\begin{enumerate}
\item[{\rm(i)}] $\|f\|_{X}=0$ implies that $f=0$ almost everywhere;

\item[{\rm(ii)}] $|g|\le|f|$ almost everywhere implies that $\|g\|_{X}\le\|f\|_{X}$;
\item[{\rm(iii)}] $0\le f_m\uparrow f$ almost everywhere on $\rn$
implies that $\|f_m\|_{X}\uparrow\|f\|_{X}$;

\item[{\rm(iv)}] $B\in\mathbb{B}$ implies that $\mathbf1_B\in X$, where
$$\mathbb{B}:=\{B(x,r):\ x\in\rn\ \ \mbox{and}\ \ r\in(0,\infty)\}.$$
\end{enumerate}
\end{definition}

Observe that, in the above definition, $\mathbb{B}$ can be replaced by the set of all bounded measurable
sets of $\rn$.

\begin{definition}\label{convex}
Let $X$ be a ball quasi-Banach function space and $p\in(0,\infty)$.
\begin{enumerate}
\item[{\rm(i)}] The \emph{$p$-convexification} $X^p$ of $X$ is defined by setting
$X^p:=\{f\in\mathbb{M}(\rn):\ |f|^p\in X\}$
equipped with the quasi-norm $\|f\|_{X^p}:=\||f|^p\|_{X}^{\frac{1}{p}}$ for any $f\in X^p$.

\item[{\rm(ii)}] The space $X$ is said to be \emph{p-convex} if there exists a positive constant $C$ such that,
for any $\{f_j\}_{j\in\mathbb{N}}\subset X^{\frac{1}{p}}$,
$$
\lf\|\sum_{j=1}^{\infty}\lf|f_j\r|\r\|_{X^{\frac{1}{p}}}\le C\sum_{j=1}^{\infty}\lf\|f_j\r\|_{X^{\frac{1}{p}}}.
$$
In particular, when $C=1$, $X$ is said to be \emph{strictly p-convex}.
\end{enumerate}
\end{definition}

Recall that Sawano et al. \cite{ykds} developed a real-variable
theory of local Hardy spaces associated with ball quasi-Banach function spaces.
More function spaces based on ball quasi-Banach function spaces can be found in \cite{s,WYYZ,ZWYY}.
Lemma \ref{ballproof} below shows that
Orlicz-slice spaces are ball quasi-Banach function spaces,
which further
implies that the local Orlicz-slice Hardy space is a special case of the
local Hardy-type space, associated with the ball quasi-Banach function space,
considered in \cite{ykds}.
To establish the finite atomic characterization of local Orlicz-slice Hardy spaces,
we still need the following three lemmas, which are just, respectively,
\cite[Lemmas 4.2, 4.3 and 4.4]{zyyw}.

\begin{lemma}\label{ballproof}
Let $t,\ q\in(0,\infty)$ and $\Phi$ be an Orlicz function with positive lower type $p_{\Phi}^-$ and positive upper type $p_{\Phi}^+$.
Then $(E_\Phi^q)_t(\rn)$ is a ball quasi-Banach function space.
\end{lemma}

\begin{lemma}\label{sconvex}
Let $t$, $q\in(0,\infty)$ and $\Phi$ be an Orlicz function with positive lower type
$p_{\Phi}^-$ and positive upper type $p_{\Phi}^+$. Let $\vartheta\in (0,\min\{p^-_{\Phi},q\}]$.
Then $(E_\Phi^q)_t(\rn)$ is a strictly $\vartheta$-convex
ball quasi-Banach function space as in Definition \ref{convex}(ii).
\end{lemma}

\begin{lemma}\label{Le818}
Let $t,\ q\in(0,\infty)$ and $\Phi$ be an Orlicz function with positive lower type $p_{\Phi}^-$
and positive upper type $p_{\Phi}^+$. Let $r\in(\max\{q,\ p_{\Phi}^+\},\infty]$ and $s\in (0,\min\{p^-_{\Phi},q\})$. Then
there exists a positive constant $C_{(s,r)}$, depending on $s$ and $r$, but independent of $t$,
such that, for any measurable function $f$,
\begin{equation}\label{jdjd}
\lf\|\mathcal{M}^{((r/s)')}(f) \r\|_{([(E_\Phi^q)_t(\rn)]^{1/s})'}\le C_{(s,r)}\lf\|f \r\|_{([(E_\Phi^q)_t(\rn)]^{1/s})'},
\end{equation}
here and thereafter, $[(E_\Phi^q)_t(\rn)]^{1/s}$ denotes the
$\frac1s$-convexification of $(E_\Phi^q)_t(\rn)$ as in Definition \ref{convex}(i)
with $X:=(E_\Phi^q)_t(\rn)$ and $p:=1/s$,
and $([(E_\Phi^q)_t(\rn)]^{1/s})'$ denotes its dual space.
\end{lemma}

Applying Lemmas \ref{main}, \ref{ballproof} and \ref{sconvex}, and  \cite[Theorem 5.3]{ykds},
we obtain the following maximal function
characterizations of the local Orlicz-slice Hardy space $(hE_\Phi^q)_t(\rn)$.

\begin{theorem}\label{mdj}
Let $a,\,b,\,t,\,q\in(0,\infty)$. Let $\Phi$ be an Orlicz function with positive lower
type $p_{\Phi}^-$ and positive upper type $p_{\Phi}^+$. Let $\varphi\in\mathcal{S}(\rn)$ satisfy
$\int_{\rn}\varphi(x)dx\neq0.$
\begin{enumerate}
\item[{\rm(i)}] Let $N\geq\lfloor b+1\rfloor$ be an integer. Then, for any $f\in\mathcal{S}'(\rn)$,
$$
\|m(f,\varphi)\|_{(E_\Phi^q)_t(\rn)}\lesssim\|m_a^*(f,\varphi)\|_{(E_\Phi^q)_t(\rn)}
\lesssim\|m_b^{**}(f,\varphi)\|_{(E_\Phi^q)_t(\rn)},
$$
$$
\|m(f,\varphi)\|_{(E_\Phi^q)_t(\rn)}\lesssim\|m_N(f)\|_{(E_\Phi^q)_t(\rn)}
\lesssim\|m_{\lfloor b+1 \rfloor}(f)\|_{(E_\Phi^q)_t(\rn)}\lesssim\|m_b^{**}(f,\varphi)\|_{(E_\Phi^q)_t(\rn)}
$$
and
$$
\|m_b^{**}(f,\varphi)\|_{(E_\Phi^q)_t(\rn)}\sim\|m_{b,\,N}^{**}(f)\|_{(E_\Phi^q)_t(\rn)},
$$
where the implicit positive constants are independent of $f$ and $t$.
\item[{\rm(ii)}] Let $b\in(\frac{2n}{\min\{p_\Phi^-,q\}},\infty)$.
Then, for any $f\in\mathcal{S}'(\rn)$,
$$
\|m_{b,\,N}^{**}(f)\|_{(E_\Phi^q)_t(\rn)}\lesssim\|m(f,\varphi)\|_{(E_\Phi^q)_t(\rn)},
$$
where the implicit positive constant is independent of $f$ and $t$. In particular, when $N\in\mathbb{N}\bigcap[\lfloor b+1\rfloor,\infty)$,
and one of the quantities
$$
\|m(f,\varphi)\|_{(E_\Phi^q)_t(\rn)},\, \|m_a^*(f,\varphi)\|_{(E_\Phi^q)_t(\rn)},\,\|m_N(f)\|_{(E_\Phi^q)_t(\rn)},
$$
$$
\|m_b^{**}(f,\varphi)\|_{(E_\Phi^q)_t(\rn)}\ \ \text{and}\ \
\|m_{b,\,N}^{**}(f)\|_{(E_\Phi^q)_t(\rn)},
$$
is finite, then the other quantities are also finite and mutually equivalent with the positive
equivalence constants
independent of $f$ and $t$.
\end{enumerate}
\end{theorem}

\begin{proof}
From Lemma \ref{ballproof}, we deduce that $(E_\Phi^q)_t(\rn)$ is a ball quasi-Banach function space.
Then (i) can be deduced directly from \cite[Theorem 5.3(i)]{ykds}.

For (ii), since $b\in(\frac{2n}{\min\{p_\Phi^-,q\}},\infty)$,
we may assume that there exists an
$r\in(0,\min\{p_\Phi^-,q\})$ such that $b>\frac{2n}{r}>\frac{2n}{\min\{p_\Phi^-,q\}}$
and let $A:=\frac{n}{r}$.
From this and Lemma \ref{sconvex}, we deduce that $(E_\Phi^q)_t(\rn)$ is strictly $r$-convex and $(b-A)r>n$.
Then, in order to prove (ii), by \cite[Theorem 5.3(ii)]{ykds}, we only need to show that, for any $f\in(E_\Phi^q)_t(\rn)$ and any $z\in\rn$,
\begin{equation}\label{q}
\lf\|\lf\{\int_{z+[0,1]^n}|f(\cdot-y)|^r\,dy\r\}
^{\frac{1}{r}}\r\|_{(E_\Phi^q)_t(\rn)}\lesssim(1+|z|)^A\|f\|_{(E_\Phi^q)_t(\rn)}.
\end{equation}
Let $\Phi_r(\tau):=\Phi(\sqrt[r]{\tau})$ for any $\tau\in[0,\infty)$. Then $\Phi_r$ is of upper type
$\frac{p_{\Phi}^+}{r}$ and of lower type $\frac{p_{\Phi}^-}{r}$, and $\frac{p_{\Phi}^-}{r}$, $\frac{q}{r}\in(1,\infty)$. Fix $z\in\rn$.
Observing that, for any $x\in\rn$,
$$
\int_{z+[0,1]^n}|f(x-y)|^r\,dy\leq
\int_{B(\vec{0}_n,|z|+\sqrt{n})}|f(x-y)|^r\,dy\lesssim\mathcal{M}([(1+|z|)^Af]^r)(x),
$$
we then have
\begin{align*}
\lf\| \lf\{ \int_{z+[0,1]^n} |f(\cdot-y)|^r \,dy \r\}^{\frac{1}{r}} \r\|_{(E_\Phi^q)_t(\rn)}
&=\lf\|\int_{z+[0,1]^n}|f(\cdot-y)|^r\,dy\r\|_{(E_{\Phi_r}^{q/r})_t(\rn)}^{\frac{1}{r}}\\
&\lesssim\lf\|\mathcal{M}([(1+|z|)^Af]^r)\r\|_{(E_{\Phi_r}^{q/r})_t(\rn)}^{\frac{1}{r}}.
\end{align*}
From this and Lemma \ref{main}, it follows that
\begin{align*}
\lf\| \lf\{ \int_{z+[0,1]^n} |f(\cdot-y)|^r \,dy \r\}^{\frac{1}{r}} \r\|_{(E_\Phi^q)_t(\rn)}
&\lesssim\lf\|\mathcal{M}([(1+|z|)^Af]^r)\r\|_{(E_{\Phi_r}^{q/r})_t(\rn)}^{\frac{1}{r}}\\
&\lesssim\lf\|[(1+|z|)^Af]^r\r\|_{(E_{\Phi_r}^{q/r})_t(\rn)}^{\frac{1}{r}}
=(1+|z|)^A\lf\| f \r\|_{(E_\Phi^q)_t(\rn)},
\end{align*}
which implies that \eqref{q} holds true and hence completes the proof of Theorem \ref{mdj}.
\end{proof}

\begin{remark}\label{214x}
By Theorem \ref{mdj}, we know that, in Definition \ref{dh}, $b$ should be large enough such that
$b\in(\frac{2n}{\min\{p_{\Phi}^-,q\}},\infty)$.
\end{remark}

\subsection{Atomic and finite atomic characterizations\label{s3.4} of $(hE_\Phi^q)_t(\rn)$}

In this subsection, we establish atomic characterizations and finite atomic characterizations
of $(hE_\Phi^q)_t(\rn)$.

In what follows, for any $L\in\mathbb{Z}_+$, the \emph{symbol}
$\mathcal{P}_L(\rn)$ denotes the set of all polynomials on $\rn$ with degree
not greater than $L$. For any $a\in L^1(\rn)$ satisfying
$$
\int_{\rn}(1+|x|)^L|a(x)|\,dx<\infty,
$$
we write $a\bot\mathcal{P}_L(\rn)$ if
$$
\int_{\rn}a(x)x^\alpha\,dx=0
$$
for any $\alpha\in\mathbb{Z}_+^n$ with $|\alpha|\leq L$.

\begin{definition}
Let $t,\ q\in(0,\infty)$, $r\in[1,\infty]$ and $d\in\mathbb{Z}_+$.
Let $\Phi$ be an Orlicz function with positive lower type $p_{\Phi}^-$ and positive upper type $p_{\Phi}^+$.
The function $a$ is called a \emph{local $((E_\Phi^q)_t(\rn),\ r,\ d)$-atom}
if there exists a cube $Q\in\mathcal{Q}$ with side length $\ell(Q)$ such that $\supp(a)\subset Q$,
$$
\|a\|_{L^r(\rn)}\le\frac{|Q|^{\frac{1}{r }}}{\|\mathbf1_Q\|_{(E_\Phi^q)_t(\rn)}},
$$
and $a\bot\mathcal{P}_d(\rn)$ when $\ell(Q)<1$.
\end{definition}

\begin{definition}\label{atomic hardy}
Let $t,\ q\in(0,\infty)$ and $\Phi$ be an Orlicz function with positive lower type
$p_{\Phi}^-$ and positive upper type $p_{\Phi}^+$. Let $r\in(\max\{1,q,\ p_{\Phi}^+\},\infty]$,
$s \in(0,\min\{p^-_{\Phi},q,1\})$ and $d\in\mathbb{Z}_+$ satisfy $d\ge\lfloor n(\frac{1}{s}-1)\rfloor$.
The \emph{atomic local Orlicz-slice Hardy space $(hE_\Phi^q)_t^{r,d}(\rn)$}
 is defined to be the set of all $f \in\mathcal{S}'(\rn)$ satisfying
that there exist a sequence $\{a_j\}_{j=1}^\infty$ of local $((E_\Phi^q)_t(\rn),\ r,\ d)$-atoms supported,
respectively, in cubes
$\{Q_j\}_{j=1}^\infty\subset\mathcal{Q}$ and a sequence $\{\lambda_j\}_{j=1}^\infty\subset[0,\infty)$
such that
\begin{equation*}
f=\sum_{j=1}^{\infty}\lambda_j a_j\ \ \mathrm{in}\ \mathcal{S}'(\rn)
\end{equation*}
and
$$
\|f\|_{(hE_\Phi^q)_t^{r,d}(\rn)}:=\inf
\lf\|\lf\{\sum_{j=1}^{\infty} \lf[\frac{\lambda_j}
{\|\mathbf1_{Q_j}\|_{(E_\Phi^q)_t(\rn)}} \r]^s\mathbf1_{Q_j} \r\}
^{\frac{1}{s}}\r\|_{(E_\Phi^q)_t(\rn)}<\infty,
$$
where the infimum is taken over all the decompositions of $f$ as above.
\end{definition}

Using Lemmas \ref{ballproof}, \ref{sconvex} and \ref{Le818},
and \cite[Theorem 4.8]{wyy19},
we immediately obtain the following atomic characterization of $(hE_\Phi^q)_t(\rn)$.

\begin{theorem}\label{atom ch}
Let all the assumptions be as in Definition \ref{atomic hardy}.
Then $(hE_\Phi^q)_t(\rn)=(hE_\Phi^q)_t^{r,d}(\rn)$
with equivalent quasi-norms.
\end{theorem}

Let $\varphi\in\mathcal{S}(\rn)$ satisfy
$\supp\varphi\subset B(\vec{0_n},1)$
and $\int_{\rn}\varphi(x)\,dx\neq0$. For any $f\in\mathcal{S}'(\rn)$ and $x\in\rn$,
we always let
$$
m_0(f)(x):=m(f,\varphi)(x),
$$
where $m(f,\varphi)(x)$ is as in Definition \ref{jdhs}(i).

Repeating an
argument similar to that used in \cite[Section 4]{t}, we immediately obtain
the following Calder\'on--Zygmund decompositions and we omit the details here.

\begin{lemma}\label{czd}
Let $t,\ q\in(0,\infty)$ and $\Phi$ be an Orlicz function with positive
lower type $p_{\Phi}^-$ and positive upper type $p_{\Phi}^+$.
Assume that $f\in(hE_\Phi^q)_t(\rn)$ and $d\in\mathbb{Z}_+$
satisfies $d\ge\lfloor n(\frac{1}{\min\{p^-_{\Phi},q\}}-1)\rfloor$. For any $j\in\zz$, let
$$
\mathcal{O}_j:=\{y\in\rn:\ m_N(f)(y)>2^j\},
$$
where $m_N(f)$ is as in Definition \ref{jdhs} with $N\in\nn$ as in Theorem \ref{mdj}.
Then the following statements hold true.
\begin{enumerate}
\item[{\rm(i)}]For any $j\in\zz$, there exists a sequence $\{Q_{j,k}^*\}_{k\in\nn}$
of cubes, which has finite intersection property, such that
$$
\mathcal{O}_j=\cup_{k\in\nn}Q_{j,k}^*.
$$
\item[{\rm(ii)}]There exist distributions $\{g_j\}_{j\in\zz}$ and $\{b_j\}_{j\in\zz}$
such that, for any $j\in\zz$, $f=g_j+b_j$ in $\mathcal{S}'(\rn)$.
\item[{\rm(iii)}]For any $j\in\zz$, the distribution $g_j$ satisfies that, for any $x\in\rn$,
\begin{equation}\label{gj}
m_0(g_j)(x)\lesssim m_N(f)(x)\mathbf1_{\mathcal{O}_j^{\complement}}(x)+\sum_{k\in\nn}
\frac{2^jl_{j,k}^{n+d+1}}{(l_{j,k}+|x-x_{j,k}|)^{n+d+1}}\mathbf1_{Q_{j,k}^*}(x),
\end{equation}
where the implicit positive  constant is independent of $f$ and $\{g_j\}_{j\in\zz}$.
Here and thereafter,
for any $j\in\zz$ and $k\in\nn$, $x_{j,k}$ and $l_{j,k}$ denote, respectively,
the center and the side length
of $Q_{j,k}^*$.
\item[{\rm(iv)}]If $f\in L^1_{\loc}(\rn)$, then, for any $j\in\zz$,
the distribution $g_j$ as in (ii) belongs to
$L^\infty(\rn)$ and $\|g_j\|_{L^\infty(\rn)}\lesssim2^j$ with the implicit
positive constant independent of $f$ and $\{g_j\}_{j\in\zz}$.
\item[{\rm(v)}]For any
$j\in\zz$,
$b_j=\sum_{k\in\nn}b_{j,k}$ in $\mathcal{S}'(\rn)$ and,
for any $k\in\nn$, $b_{j,k}:=(f-c_{j,k})\eta_{j,k}$ if $l_{j,k}<1$
and $b_{j,k}:=f\eta_{j,k}$ if $l_{j,k}\geq1$, where $\{\eta_{j,k}\}_{k\in\nn}$ is a partition
of unity with respect to $\{Q_{j,k}^*\}_{k\in\nn}$, namely, for any $k\in\nn$,
$\eta_{j,k}\in \mathcal{S}(\rn)$, $\supp(\eta_{j,k})\subset Q_{j,k}^*$, $0\leq\eta_{j,k}\leq1$,
and
$$
\sum_{k\in\nn}\eta_{j,k}=\mathbf1_{\mathcal{O}_j},
$$
$c_{j,k}$ is a polynomial satisfying, for any $q\in\mathcal{P}_d(\rn)$,
$$
\la f-c_{j,k},q\eta_{j,k}\ra=0.
$$
Moreover, for any $j\in\zz$, $k\in\nn$ and $x\in\rn$,
$$
m_0(b_{j,k})(x)\lesssim m_N(f)(x)\mathbf1_{Q_{j,k}^*}(x)+\frac{2^j
l_{j,k}^{n+d+1}}{|x-x_{j,k}|^{n+d+1}}\mathbf1_{Q_{j,k}^{*\complement}}(x),
$$
where the implicit positive  constant is independent of $f$, $k$ and $j$.
\end{enumerate}
\end{lemma}

\begin{theorem}\label{decomposition}
Let $t,\,q\in(0,\infty)$ and $\Phi$ be an Orlicz function with positive
lower type $p_{\Phi}^-$ and positive upper type $p_{\Phi}^+$.
Let $r\in(\max\{1,\,q,\,p_{\Phi}^+\},\infty]$,
$s\in (0,\min\{p^-_{\Phi},q,1\})$ and $d\in\mathbb{Z}_+$ satisfy $d\geq\lfloor n(\frac{1}{\min\{p^-_{\Phi},q,1\}}-1)\rfloor$. Then, for any
$f\in (hE_\Phi^q)_t(\rn)\cap L_{\loc}^1(\rn)$,
there exist a sequence $\{a_{j,k}\}_{j\in\zz,k\in\nn}$
of local $((E_\Phi^q)_t,\infty,d)$-atoms supported, respectively, in cubes
$\{Q_{j,k}\}_{j\in\zz,k\in\nn}\subset\mathcal{Q}$, and a sequence $\{\lambda_{j,k}\}_{j\in\zz,k\in\nn}\subset[0,\infty)$
such that
\begin{equation}\label{convergence2}
f=\sum_{j\in\zz,k\in\nn}\lambda_{j,k} a_{j,k}
\end{equation}
converges both in $\mathcal{S}'(\rn)$ and almost everywhere and,
for any $j\in\zz$,
\begin{equation}\label{55}
\mathcal{O}_j=\cup_{k\in\nn}Q_{j,k}\ and\
\{Q_{j,k}\}_{k\in\nn}\ has\ finite\ intersection\ property,
\end{equation}
where $\mathcal{O}_j:=\{x\in\rn:\ m_N(f)(x)>2^j\}$.
Moreover, for any $j\in\zz$, $k\in\nn$
and almost every $x\in\rn$,
\begin{equation}\label{66}
|\lambda_{j,k} a_{j,k}(x)|\lesssim2^j
\end{equation}
and
\begin{equation}\label{88}
\lf\|\lf\{\sum_{j\in\zz,k\in\nn} \lf[\frac{\lambda_{j,k}}
{\|\mathbf1_{Q_{j,k}}\|_{(E_\Phi^q)_t(\rn)}} \r]^s\mathbf1_{Q_{j,k}} \r\}
^{\frac{1}{s}}\r\|_{(E_\Phi^q)_t(\rn)}\lesssim\|f\|_{(hE_\Phi^q)_t(\rn)},
\end{equation}
where the implicit positive constant is independent of $f$, $x$ and $t$.
\end{theorem}

\begin{proof}
Assume that $f\in L_{\loc}^1(\rn)\cap(hE_\Phi^q)_t(\rn)$.
Then, by Lemma \ref{czd}, we know that,
for any $j\in\zz$, there exist two distributions $g_j$
and $b_j$ such that
$$
f=g_j+b_j  \quad\mathrm{in}\ \mathcal{S}'(\rn)
$$
and $b_j=\sum_{k\in\nn}b_{j,k}$ in $\mathcal{S}'(\rn)$, where, for any $j\in\zz$
and $k\in\nn$,
$$
b_{j,k}:=\begin{cases}\displaystyle\ (f-c_{j,k})\eta_{j,k}\quad &\mathrm{if}\ l_{j,k}\in(0,1),\\
\ f\eta_{j,k}\quad &\mathrm{if}\ l_{j,k}\in[1,\infty)
\end{cases}
$$
supported in $Q_{j,k}^*$, and $Q_{j,k}^*$, $l_{j,k}$, $c_{j,k}$ and $\eta_{j,k}$ are as in Lemma \ref{czd}.
From this, Theorem \ref{mdj}(ii) and Lemma \ref{czd}(v),
we deduce that
\begin{align}\label{c11}
\lf\|f-g_j\r\|_{(hE_\Phi^q)_t(\rn)}&\sim
\lf\|\sum_{k\in\nn}b_{j,k}\r\|_{(hE_\Phi^q)_t(\rn)}\lesssim
\lf\|\sum_{k\in\nn}m_0(b_{j,k})\r\|_{(E_\Phi^q)_t(\rn)}\\\noz
&\lesssim
\lf\|\sum_{k\in\nn}m_N(f)\mathbf1_{Q_{j,k}^*}\r\|_{(E_\Phi^q)_t(\rn)}
+2^j\lf\|\sum_{k\in\nn}\frac{
l_{j,k}^{n+d+1}}{|\cdot-x_{j,k}|^{n+d+1}}\mathbf1_{Q_{j,k}^{*\complement}}\r\|_{(E_\Phi^q)_t(\rn)}\\\noz
&\lesssim\lf\|\sum_{k\in\nn}m_N(f)\mathbf1_{Q_{j,k}^*}\r\|_{(E_\Phi^q)_t(\rn)}
+2^j\lf\|\lf\{\sum_{k\in\nn}\lf[\mathcal{M}(\mathbf1_{Q_{j,k}^*})\r]
^{\frac{1}{u}}\r\}^{u}\r\|_{(E_{\Phi_u}^{\frac{q}{u}})_t(\rn)},
\end{align}
where $u:=\frac{n}{n+d+1}$ and $\Phi_u(\tau):=\Phi(\sqrt[u]{\tau})$ for any $\tau\in[0,\infty)$. Notice that $\min\{p^-_{\Phi},q\}>u$
and $\Phi_u$ is of lower type $\frac{p_{\Phi}^-}{u}\in(1,\infty)$.
By this, \eqref{c11} and Lemma \ref{main}, we conclude that
\begin{align}\label{340}
\lf\|f-g_j\r\|_{(hE_\Phi^q)_t(\rn)}
&\lesssim\lf\|\sum_{k\in\nn}m_N(f)\mathbf1_{Q_{j,k}^*}\r\|_{(E_\Phi^q)_t(\rn)}
+2^j\lf\|\sum_{k\in\nn}\mathbf1_{Q_{j,k}^*}
\r\|_{(E_{\Phi}^{q})_t(\rn)}\\ \noz
&\lesssim
\lf\|m_N(f)\mathbf1_{\mathcal{O}_{j}}\r\|_{(E_\Phi^q)_t(\rn)}\to0
\end{align}
as $j\to\infty$.
In addition, from the fact that $f\in L_{\loc}^1(\rn)$ and Lemma \ref{czd},
it follows that, for any $j\in\zz$,
$\|g_j\|_{L^\infty(\rn)}\lesssim2^j$.
This, together with \eqref{340}, further implies that
\begin{equation}\label{341}
f=\sum_{j\in\zz}\lf(g_{j+1}-g_{j}\r) \quad \mathrm{in}\ \mathcal{S}'(\rn).
\end{equation}
Moreover, since, for any $j\in\zz$, $\supp(b_j)\subset\mathcal{O}_j$ and $|\mathcal{O}_j|\to0$
as $j\to\infty$, if
letting $\mathcal{O}:=\cap_{j\in\zz}\mathcal{O}_j$,
then it follows that $|\mathcal{O}|=0$.
Moreover, for any $x\in\rn\setminus\mathcal{O}$,
by $\mathcal{O}_j\supset\mathcal{O}_{j+1}$
and $\supp(b_j)\subset\mathcal{O}_j$ for
any $j\in\zz$, we know that there exists a
$j_0\in\zz$ such, for any $j\geq j_0$, $b_j(x)=0$.
Thus, as $j\to\infty$, $b_j\to0$ almost everywhere,
which, together with Lemma \ref{czd}(ii), further
implies
that, as $j\to\infty$, $g_j\to f$ almost everywhere.
Thus, \eqref{341} also holds true almost everywhere.

For any $j\in\zz$, $k\in\nn$ and $q\in\mathcal{P}_d(\rn)$, let
$F_j:=\{m\in\nn:\ l_{j,m}<1\}$,
\begin{equation}\label{342}
\|q\|_{j,k}:=\lf[\frac{1}{\int_{\rn}\eta_{j,k}(x)\,dx}\int_{\rn}|q(x)|^2
\eta_{j,k}(x)\,dx\r]^{1/2}
\end{equation}
and, for any $i,\ k\in\nn$ and $j\in\zz$, $c_{j+1,k,i}$ be the orthogonal projection of
$(f-c_{j+1,i})\eta_{j,k}$ on $\mathcal{P}_d(\rn)$ with respect to the norm in \eqref{342}.
Then, by the proof of \cite[Lemma 5.4]{t} (see also \cite[Lemma 5.4]{yy}), we know
that
\begin{equation}\label{31}
f=\sum_{j\in\zz}\lf(g_{j+1}-g_{j}\r)=\sum_{j\in\zz}\sum_{k\in\nn}
\lf[b_{j,k}-\sum_{i\in\nn}b_{j+1,i}\eta_{j,k}+\sum_{i\in F_{j+1}}c_{j+1,k,i}\eta_{j+1,i}\r]
=:\sum_{j\in\zz}\sum_{k\in\nn}A_{j,k}
\end{equation}
converges in $\mathcal{S}'(\rn)$ and almost everywhere.
Moreover, by the proof of \cite[Lemma 5.4]{t}, we also know that, for any $j\in\zz$
and $k\in\nn$, $A_{j,k}$ is supported in $Q_{j,k}:=C_0Q_{j,k}^*$
satisfying $\|A_{j,k}\|_{L^\infty(\rn)}\leq C_12^j$ with $C_0$ and $C_1$
being some positive constants independent of $j$ and $k$,
and
\begin{equation}\label{77}
\{Q_{j,k}\}_{k\in\nn}\ \mathrm{has\ finite\ intersection\ property\ satisfying}\
\mathcal{O}_j=\cup_{k\in\nn}Q_{j,k},
\end{equation}
and,
when $\ell(Q_{j,k})<1$, for any $q\in\mathcal{P}_d(\rn)$,
$$
\int_{\rn}A_{j,k}(x)q(x)\,dx=0,
$$
where $\ell(Q_{j,k})$ denotes the side length
of $Q_{j,k}$.
For any $j\in\zz$ and $k\in\nn$, let
\begin{equation}\label{343}
\lambda_{j,k}:=C_02^j\|\mathbf1_{Q_{j,k}}\|_{(E_{\Phi}^{q})_t(\rn)}
\quad \mathrm{and}\quad a_{j,k}:=\frac{A_{j,k}}{\lambda_{j,k}}.
\end{equation}
Then it is easy to see that, for any $j\in\zz$ and $k\in\nn$,
\begin{equation}\label{111}
|\lambda_{j,k} a_{j,k}|\lesssim2^j
\end{equation}
and $a_{j,k}$ is a local $((E_\Phi^q)_t(\rn),\infty,d)$-atom.
Moreover, we have
$$
f=\sum_{j\in\zz}\sum_{k\in\nn}\lambda_{j,k}a_{j,k}
$$
converges in $\mathcal{S}'(\rn)$ and almost everywhere.
From this, \eqref{77} and \eqref{111}, we deduce that \eqref{convergence2},
\eqref{55} and \eqref{66} of Theorem \ref{decomposition} hold true.

In addition, from \eqref{343}, the fact that $\cup_{k\in\nn}Q_{j,k}=\mathcal{O}_j$,
the finite intersection
property of $\{Q_{j,k}\}_{k\in\nn}$ and the definition of $\mathcal{O}_j$, we further
deduce that, for any given $s\in (0,\min\{p^-_{\Phi},q,1\})$,
\begin{align*}
&\lf\|\lf\{\sum_{j\in\zz}\sum_{k\in\nn} \lf[\frac{\lambda_{j,k}\mathbf1_{Q_{j,k}}}
{\|\mathbf1_{Q_{j,k}}\|_{(E_\Phi^q)_t(\rn)}} \r]^s \r\}
^{\frac{1}{s}}\r\|_{(E_\Phi^q)_t(\rn)}\\
&\quad\sim\lf\|\lf\{\sum_{j\in\zz}\sum_{k\in\nn} 2^{js}\mathbf1_{Q_{j,k}} \r\}
^{\frac{1}{s}}\r\|_{(E_\Phi^q)_t(\rn)}
\sim\lf\|\lf\{\sum_{j\in\zz} \lf(2^{j}\mathbf1_{\mathcal{O}_j}\r)^s \r\}
^{\frac{1}{s}}\r\|_{(E_\Phi^q)_t(\rn)}\\
&\quad\sim\lf\|\sum_{j\in\zz} 2^{j}\mathbf1_{\mathcal{O}_j\setminus
\mathcal{O}_{j+1}}
\r\|_{(E_\Phi^q)_t(\rn)}
\sim\lf\|m_N(f)\r\|_{(E_\Phi^q)_t(\rn)}\sim\lf\|f\r\|_{(hE_\Phi^q)_t(\rn)}.
\end{align*}
This implies that, for any $f\in L_{\loc}^1(\rn)\cap(hE_\Phi^q)_t(\rn)$,
\eqref{88} holds true and hence finishes the proof of
Theorem \ref{decomposition}.
\end{proof}

\begin{definition}\label{definite}
Let $t,\ q\in(0,\infty)$ and $\Phi$ be an Orlicz function with positive lower type $p_{\Phi}^-$ and positive upper type $p_{\Phi}^+$. Let $r\in(\max\{1,q,\ p_{\Phi}^+\},\infty]$, $s \in(0,\min\{p^-_{\Phi},q,1\}]$ and $d\in\mathbb{Z}_+$ satisfy
$d\ge\lfloor n(\frac{1}{s}-1)\rfloor$. The \emph{finite atomic Orlicz-slice Hardy space} $(hE_\Phi^{q,r,d})_t^{\fin}(\rn)$ is defined to be the set of all finite linear combinations of
local $((E_\Phi^q)_t(\rn),\ r,\ d)$-atoms.
The quasi-norm $\|\cdot\|_{(hE_\Phi^{q,r,d})_t^{\fin}(\rn)}$ in
$(hE_\Phi^{q,r,d})_t^{\fin}(\rn)$ is defined by
setting, for any $f\in(hE_\Phi^{q,r,d})_t^{\fin}(\rn)$
\begin{align*}
&\|f\|_{(hE_\Phi^{q,r,d})_t^{\fin}(\rn)}\\
&\quad :=\inf\lf\{\lf\|\lf\{\sum_{j=1}^{m} \lf[\frac{\lambda_j}
{\|\mathbf1_{Q_j}\|_{(E_\Phi^q)_t(\rn)}} \r]^s\mathbf1_{Q_j} \r\}
^{\frac{1}{s}}\r\|_{(E_\Phi^q)_t(\rn)}:\ m\in\mathbb{N},\ f=\sum_{j=1}^{m}\lambda_j a_j,\ \{\lambda_j\}_{j=1}^m\subset[0,\infty)\r\},
\end{align*}
where the infimum is taken over all the above finite linear combinations of $f$
via a sequence
of local $((E_\Phi^q)_t(\rn),\ r,\ d)$-atoms $\{a_j\}_{j=1}^m$ supported, respectively,
in cubes $\{Q_j\}_{j=1}^m$.
\end{definition}

Then we have the following conclusion. In what follows, the \emph{symbol} $\mathcal{C}{(\rn)}$ is defined to be the
set of all continuous complex-valued functions on $\rn$.

\begin{theorem}\label{finite}
Let $t,\ q\in(0,\infty)$ and $\Phi$ be an Orlicz function with positive lower type $p_{\Phi}^-$ and positive upper type $p_{\Phi}^+$. Let $r\in(\max\{1,q,\ p_{\Phi}^+\},\infty]$, $s\in(0,\min\{p^-_{\Phi},q,1\})$ and $d\in\mathbb{Z}_+$ satisfy
$d\ge\lfloor n(\frac{1}{s}-1)\rfloor$.
\begin{enumerate}
\item[{\rm(i)}] If $r\in(\max\{1,q,\ p_{\Phi}^+\},\infty)$, then $\|\cdot\|_{(hE_\Phi^q)_t(\rn)}$ and $\|\cdot\|_{(hE_\Phi^{q,r,d})_t^{\fin}(\rn)}$ are equivalent on
     the space $(hE_\Phi^{q,r,d})_t^{\fin}(\rn)$ with the positive equivalence constants independent of $t$.
\item[{\rm(ii)}] If $r=\fz$, then $\|\cdot\|_{(hE_\Phi^q)_t(\rn)}$ and $\|\cdot\|_{(hE_\Phi^{q,\infty,d})_t^{\fin}(\rn)}$ are equivalent on $(hE_\Phi^{q,\infty,d})_t^{\fin}(\rn)\cap\mathcal{C}{(\rn)}$ with the positive equivalence constants independent of $t.$
\end{enumerate}
\end{theorem}

To show Theorem \ref{finite}, we need the following lemma. Since the proof is similar to that of
\cite[Lemma 4.6]{zyyw}, we omit the details.
\begin{lemma}\label{youjie}
Let $t,\ q\in(0,\infty)$ and $\Phi$ be an Orlicz function with positive lower type $p_{\Phi}^-$ and positive upper type $p_{\Phi}^+$.
Let $N\in\mathbb{N}\cap(\lfloor\frac{n}{\min\{p_\Phi^-,q\}}+1\rfloor,\infty)$.
Suppose $f\in(hE_\Phi^q)_t(\rn)\cap L^1_\loc(\rn)$,
$\|f\|_{(hE_\Phi^q)_t(\rn)}=1$ and $\supp(f)\subset B(\vec{0}_n,R)$ with $R\in(1,\infty)$. Then
there exists a positive constant $C_{(N)}$, depending on $N$, but independent of $f$ and $t$, such that,
for any $x\notin B(\vec{0}_{n},4R)$,
\begin{equation}\label{423}
m_N(f)(x)\le  C_{(N)}\lf\|\mathbf1_{B(\vec{0}_n,R)}\r\|_{(E_\Phi^q)_t(\rn)}^{-1}.
\end{equation}
\end{lemma}

\begin{proof}[Proof of Theorem \ref{finite}]
Obviously, from Theorem \ref{atom ch}, we deduce that
$$(hE_\Phi^{q,r,d})_t^{\fin}(\rn)\subset(hE_\Phi^q)_t(\rn)$$
and, for any $f\in(hE_\Phi^{q,r,d})_t^{\fin}(\rn)$,
$$
\|f\|_{(hE_\Phi^q)_t(\rn)}\lesssim\|f\|_{(hE_\Phi^{q,r,d})_t^{\fin}(\rn)}.
$$
Thus, to complete the proof of this theorem,
we still need to show that, for any given $t,\ q,\ d$ as in Theorem \ref{finite}(i) and
$r\in(\max\{1,\ q,\ p_{\Phi}^+\},\infty)$, and any $f\in(hE_\Phi^{q,r,d})_t^{\fin}(\rn)$,
$$
\|f\|_{(hE_\Phi^{q,r,d})_t^{\fin}(\rn)}\lesssim\|f\|_{(hE_\Phi^q)_t(\rn)},
$$
and that a similar estimate also holds true for $r=\infty$, any given $t,\ q,\ d$ as in
Theorem \ref{finite}(ii),
and any $f\in(hE_\Phi^{q,\infty,d})_t^{\fin}(\rn)\cap\mathcal{C}{(\rn)}$.

Let $r\in(\max\{1,\ q,\ p_{\Phi}^+\},\infty]$. By the homogeneity of both $\|\cdot\|_{(hE_\Phi^{q,r,d})_t^{\fin}(\rn)}$ and $\|\cdot\|_{(hE_\Phi^q)_t(\rn)}$,
without loss of generality,
we may also assume that $f\in(hE_\Phi^{q,r,d})_t^{\fin}(\rn)$ and $\|f\|_{(hE_\Phi^q)_t(\rn)}=1$.
Since $f$ is a finite linear combination of $((E_\Phi^q)_t(\rn),\ r,\ d)$-atoms,
it follows that there exists an $R\in(1,\infty)$ such that
$f$ is supported in $B(\vec{0}_n,R)$. Thus, if letting $N$ be as in Lemma \ref{youjie}, then, by Lemma \ref{youjie},
we know that
there exists a positive constant $C_{(N)}$ such that, for any $x\notin B(\vec{0}_{n},4R)$,
\begin{equation}\label{78}
m_N(f)(x)\le  C_{(N)}\lf\|\mathbf1_{B(\vec{0}_n,R)}\r\|_{(E_\Phi^q)_t(\rn)}^{-1}.
\end{equation}
For any $j\in\mathbb{Z}$, let $\mathcal{O}_j:=\lf\{x\in\rn:\ m_N(f)(x)>2^j\r\}$. Denote by
$j'$ the largest integer $j$ such that
\begin{equation}\label{130}
2^{j'}<C_{(N)}\lf\|\mathbf1_{B(\vec{0}_n,R)}\r\|_{(E_\Phi^q)_t(\rn)}^{-1}.
\end{equation}
Then, by (\ref{78}), we find that,
for any $j\in\{j'+1,\ j'+2,\ \ldots\}$,
\begin{equation}\label{oo}
\mathcal{O}_j\subset B(\vec{0}_n,4R).
\end{equation}
Notice that $f\in L^{\widetilde{r}}(\rn)$, where $\widetilde{r}:=r$ when
$r\in(\max\{1,\ q,\ p_{\Phi}^+\},\infty)$,  and $\widetilde{r}:=2$ when
$r=\infty$.
Then, by Theorem \ref{decomposition},
we know that there exist a sequence $\{(a_{j,k},Q_{j,k})\}_{j\in\mathbb{Z},k\in \nn}$ of pairs of
$((E_\Phi^q)_t(\rn),\ \infty,\ d)$-atoms and their supports, and a sequence of scalars, $\{\lambda_{j,k}\}_{j\in\mathbb{Z},k\in \nn}\subset[0,\infty)$, such that
\begin{equation}\label{eqq}
f=\sum_{j\in\zz}\sum_{k\in \nn}\lambda_{j,k}a_{j,k}
\end{equation}
both in $\mathcal{S}'(\rn)$ and almost everywhere,
where, for any given $j\in\zz$, $\{Q_{j,k}\}_{k\in \nn}$ is a
family of cubes with finite intersection
property such that $\mathcal{O}_j=\cup_{k\in\nn}Q_{j,k}$.
Moreover, we have, for any given $s\in(0,\min\{p^-_{\Phi},q,1\})$,
\begin{equation}\label{331}
\lf\|\lf\{\sum_{j\in\zz}\sum_{k\in \nn} \lf[\frac{\lambda_{j,k}}
{\|\mathbf1_{Q_{j,k}}\|_{(E_\Phi^q)_t(\rn)}} \r]^s\mathbf1_{Q_{j,k}} \r\}
^{\frac{1}{s}}\r\|_{(E_\Phi^q)_t(\rn)}\lesssim\|f\|_{(hE_\Phi^q)_t(\rn)}.
\end{equation}
Define
\begin{equation}\label{eq2}
h:=\sum_{j=-\infty}^{j'}\sum_{k\in \nn}\lambda_{j,k}a_{j,k}\quad\mathrm{and}
\quad l:=\sum_{j=j'+1}^{\infty}\sum_{k\in \nn}\lambda_{j,k}a_{j,k},
\end{equation}
where the series converge both in $\mathcal{S}'(\rn)$ and almost everywhere. Clearly $f=h+l$ and,
by (\ref{oo}), $\supp(l)\subset\cup_{j=j'+1}^\infty\mathcal{O}_j\subset B(\vec{0}_n,4R)$. Therefore,
$f=l=0$ on
$\rn\setminus B(\vec{0}_n,4R)$ and hence $\supp(h)\subset B(\vec{0}_n,4R)$. Moreover, by Theorem \ref{decomposition}, we know that there exists a positive constant $C_0$ such that
$\|\lambda_{j,k}a_{j,k}\|_{L^{\infty}(\rn)}\le C_02^j$.
Since $f\in L^{\widetilde{r}}(\rn)$, it follows that
\begin{align}\label{ll}
\|l\|_{L^{\widetilde{r}}(\rn)}&\leq\lf\|\sum_{j=j'+1}^{\infty}\sum_{k\in \nn}|\lambda_{j,k}a_{j,k}|\r\|_{L^{\widetilde{r}}(\rn)}\lesssim\lf\|\sum_{j=j'+1}^{\infty}\sum_{k\in \nn}2^j\mathbf1_{Q_{j,k}}\r\|_{L^{\widetilde{r}}(\rn)}\\ \noz
&\lesssim\lf\|\sum_{j=j'+1}^{\infty}2^j
\mathbf1_{\mathcal{O}_j}\r\|_{L^{\widetilde{r}}(\rn)}\lesssim\|m_N(f)\|_{L^{\widetilde{r}}(\rn)}
\lesssim\|f\|_{L^{\widetilde{r}}(\rn)}.
\end{align}
Thus, $l\in L^{\widetilde{r}}(\rn)$ and so $h=f-l\in L^{\widetilde{r}}(\rn)$.
In order to estimate the size of $h$ in $B(\vec{0}_n,4R)$, recall that
\begin{equation}\label{129}
\lf\|\lambda_{j,k}a_{j,k}\r\|_{L^{\infty}(\rn)}\lesssim2^j,\ \supp(a_{j,k})\subset Q_{j,k}
\quad\mathrm{and}\quad\sum_{k\in\nn}\mathbf1_{Q_{j,k}}\lesssim1.
\end{equation}
From \cite[(4.21)]{zyyw}, it follows that
$
\|\mathbf1_{B(\vec{0}_n,R)}\|_{(E_\Phi^q)_t(\rn)}
\sim\|\mathbf1_{Q(\vec{0}_n,8R)}\|_{(E_\Phi^q)_t(\rn)}.
$
Combining this, (\ref{130}) and (\ref{129}), we conclude that
\begin{equation*}
\|h\|_{L^{\infty}(\rn)}\le\sum_{j\le j'}\lf\|\sum_{k\in \nn}|\lambda_{j,k}a_{j,k}|\r\|_{L^{\infty}(\rn)}
\lesssim\sum_{j\le j'}2^j\lesssim2^{j'}\lesssim\lf\|\mathbf1_{B(\vec{0}_n,R)}\r\|_{(E_\Phi^q)_t(\rn)}^{-1}
\lesssim\lf\|\mathbf1_{Q(\vec{0}_n,8R))}\r\|_{(E_\Phi^q)_t(\rn)}^{-1}.
\end{equation*}
From this,
together with the fact that $\supp(h)\subset Q(\vec{0}_n,8R)$, it follows that
there exists a positive constant $\widetilde{C}$ such that $\widetilde{C}^{-1}h$
is a local $((E_\Phi^q)_t(\rn),\infty,d)$-atom.

Now, to complete the proof of Theorem \ref{finite}(i),
we may assume that $r\in(\max\{1,\ q,\ p_{\Phi}^+\},\infty)$.
We rewrite $l$ as a finite linear
combination of $((E_\Phi^q)_t(\rn),\ r,\ d)$-atoms. For any $i\in\mathbb{N}$, let
$$
F_i:=\lf\{(j,k)\in\mathbb{Z}\times\mathbb{Z}_+:\ j\in\lf\{j'+1,\ j'+2,\ \ldots\r\},\ k\in \nn,\ |j|+k\le i\r\},
$$
and $l_i:=\sum_{(j,k)\in F_i}\lambda_{j,k}a_{j,k}$. Since the series
$l=\sum_{j=j'+1}^{\infty}\sum_{k\in \nn}\lambda_{j,k}a_{j,k}$ converges in $L^r(\rn)$,
it follows that there exists a positive integer $i_0$, which may depend on $t$ and $f$,
such that
$$\|l-l_{i_0}\|_{L^r(\rn)}\le\frac{|Q(\vec{0}_n,8R)|^{\frac{1}{r }}}{\|\mathbf1_{Q(\vec{0}_n,8R)}\|_{(E_\Phi^q)_t(\rn)}}.$$
Thus, $l-l_{i_0}$ is a local $((E_\Phi^q)_t(\rn),\ r,\ d)$-atom because $\supp(l-l_{i_0})\subset B(\vec{0}_n,4R)\subset Q(\vec{0}_n,8R)$. Therefore,
$$
f=h+l=\widetilde{C}\widetilde{C}^{-1}h+(l-l_{i_0})+l_{i_0}
$$
is a finite linear combination of $((E_\Phi^q)_t(\rn),\ r,\ d)$-atoms. Moreover, by (\ref{331}),
we have, for any given $s\in(0,\min\{p^-_{\Phi},q,1\})$,
\begin{align*}
\|f\|_{(hE_\Phi^{q,r,s})_t^{\fin}(\rn)}
&\le\lf\|\lf\{\lf[\frac{\widetilde{C}}
{\|\mathbf1_{Q(\vec{0}_n,8R)}\|_{(E_\Phi^q)_t(\rn)}} \r]^s\mathbf1_{Q(\vec{0}_n,8R)}+\lf[\frac{1}
{\|\mathbf1_{Q(\vec{0}_n,8R)}\|_{(E_\Phi^q)_t(\rn)}} \r]^s\mathbf1_{Q(\vec{0}_n,8R)}\r.\r.\\
&\quad\lf.\lf.+\sum_{(j,k)\in F_{i_0}}\lf[\frac{\lambda_{j,k}}
{\|\mathbf1_{Q_{j,k}}\|_{(E_\Phi^q)_t(\rn)}} \r]^s\mathbf1_{Q_{j,k}} \r\}
^{\frac{1}{s}}\r\|_{(E_\Phi^q)_t(\rn)}\\
&\lesssim1+\lf\|\lf\{\sum_{(j,k)\in F_{i_0}}\lf[\frac{\lambda_{j,k}}
{\|\mathbf1_{Q_{j,k}}\|_{(E_\Phi^q)_t(\rn)}} \r]^s\mathbf1_{Q_{j,k}} \r\}
^{\frac{1}{s}}\r\|_{(E_\Phi^q)_t(\rn)}\\
&\lesssim1+\lf\|\lf\{\sum_{j\in\zz}\sum_{k\in\nn} \lf[\frac{\lambda_{j,k}}
{\|\mathbf1_{Q_{j,k}}\|_{(E_\Phi^q)_t(\rn)}} \r]^s\mathbf1_{Q_{j,k}} \r\}
^{\frac{1}{s}}\r\|_{(E_\Phi^q)_t(\rn)}\lesssim1+\|f\|_{(hE_\Phi^q)_t(\rn)}\lesssim1.
\end{align*}
Thus, $\|f\|_{(hE_\Phi^{q,r,d})_t^{\fin}(\rn)}\lesssim1$.
This finishes the proof of Theorem \ref{finite}(i).

To prove Theorem \ref{finite}(ii), we assume that
$f\in(hE_\Phi^{q,\infty,d})_t^{\fin}(\rn)\cap\mathcal{C}(\rn)$ and
$\|f\|_{(hE_\Phi^q)_t(\rn)}=1$. Since $f$ has a compact support, it follows that $f$ is uniformly continuous.
Then, by \eqref{31} and the proof
of \cite[Lemma 5.4]{t}, we know that each
local $((E_\Phi^q)_t(\rn),\ \infty,\ d)$-atom $a_{j,k}$
in \eqref{eqq} is continuous. Since $f$ is bounded, from the boundedness of $m_N(f)$ on
$L^{\infty}(\rn)$, it follows that there exists a
positive integer $j''>j'$ such that $\mathcal{O}_j=\emptyset$ for
any $j\in\{j''+1,\ j''+2,\ \ldots\}$. Consequently,
in this case, $l$ in \eqref{eq2} becomes
$$
l=\sum_{j=j'+1}^{j''}\sum_{k\in \nn}\lambda_{j,k}a_{j,k}.
$$

Let $\epsilon\in(0,\infty)$. Since $f$ is uniformly continuous, it follows that there exists a
$\delta\in(0,\infty)$ such that, when $|x-y|<\delta$, then $|f(x)-f(y)|<\epsilon$.
Without loss of generality, we may assume that $\delta<1$.
Write $l=l_1^\epsilon+l_2^\epsilon$ with
$l_1^\epsilon:=\sum_{(j,k)\in F_1}\lambda_{j,k}a_{j,k}$ and
$l_2^\epsilon:=\sum_{(j,k)\in F_2}\lambda_{j,k}a_{j,k}$, where
$$F_1:=\lf\{(j,k)\in\mathbb{Z}\times\mathbb{Z}_+:\ j\in\lf\{j'+1,\ \ldots,\ j''\r\},\ k\in \nn,
\ {\rm diam}\,(Q_{j,k})\ge\delta\r\}$$
and
$$F_2:=\lf\{(j,k)\in\mathbb{Z}\times\mathbb{Z}_+:\ j\in\lf\{j'+1,\ \ldots,\ j''\r\},\ k\in \nn,\ {\rm
diam}\,(Q_{j,k})<\delta\r\}.$$
Observe that $l_1^\epsilon$ is a finite summation.
Since the atoms are continuous, we know that
$l_1^\epsilon$ is also a continuous function.

For $l_2^\epsilon$, similarly to
the proof of \cite[pp.\,44-45]{yy}, we conclude that
$$\|l_2^\epsilon\|_{L^\infty(\rn)}\lesssim(j''-j')\epsilon.$$
This means that one can write $l$ as a sum of one continuous term and one which is uniformly arbitrarily small.
Thus, $l$ is continuous and so is $h=f-l$.

To find a finite atomic decomposition of $f$, we use again the splitting $l=l_1^\epsilon+l_2^\epsilon$.
It is clear that, for any $\epsilon\in(0,\infty)$, $l_1^\epsilon$ is a finite linear combination of continuous
$((E_\Phi^q)_t(\rn),\ \infty,\ d)$-atoms. Also, since both $l$ and
$l_1^\epsilon$ are continuous, it follows that
$l_2^\epsilon=l-l_1^\epsilon$ is also continuous.
Moreover, $\supp(l_2^\epsilon)\subset B(\vec{0}_n,4R)\subset Q(\vec{0}_n,8R)$ and $ \|l_2^\epsilon\|_{L^\infty(\rn)}\lesssim(j''-j')\epsilon$.
So we can choose $\epsilon$ small enough such that $l_2^\epsilon$ becomes an arbitrarily small multiple
of a continuous local $((E_\Phi^q)_t(\rn),\ \infty,\ d)$-atom. Therefore, $f=h+l_1^\epsilon+l_2^\epsilon$ is
a finite linear continuous atomic combination. Then, by an argument similar to the proof of (i), we obtain
$\|f\|_{(hE_\Phi^{q,\infty,d})_t^{\fin}(\rn)}\lesssim1$. This finishes the proof of (ii) and hence
of Theorem \ref{finite}.
\end{proof}

\subsection{Dual spaces of $(hE_\Phi^q)_t(\rn)$\label{s3.5}}

In this subsection, we give the dual space of the
local Orlicz-slice Hardy space $(hE_\Phi^q)_t(\rn)$ with $\max\{p_{\Phi}^+,\ q\}\in(0,1]$.

We first recall the following notions of both
Orlicz-slice Hardy spaces and Campanato spaces related to
Orlicz-slice spaces, which were introduced in \cite{zyyw}. To this end,
let us begin with the following notion of radial maximal functions.

\begin{definition}
Let $\varphi\in\mathcal{S}(\rn)$.
The \emph{radial maximal function} $M(f,\varphi)$ of any $f\in\mathcal{S}'(\rn)$
is defined by setting, for any $x\in\rn$,
$$M(f,\varphi)(x):=\sup_{s\in(0,\infty)}\lf|(\varphi_s \ast f)(x)\r|.$$
\end{definition}

\begin{definition}
Let $t$, $q\in(0,\infty) $ and $\Phi$ be an Orlicz function with positive lower type $p_{\Phi}^-$ and positive upper type $p_{\Phi}^+$.
Then the \emph{Orlicz-slice Hardy space $(HE_\Phi^q)_t(\rn)$} is defined by setting
$$
(HE_\Phi^q)_t(\rn):=\lf\{f\in\mathcal{S}'(\rn):\ \|f\|_{(HE_\Phi^q)_t(\rn)}:=
\|M(f,\varphi)\|_{(E_\Phi^q)_t(\rn)}<\infty \r\},
$$
where $\varphi\in\mathcal{S}(\rn)$ satisfies
$
\int_{\rn}\varphi(x)\,dx\neq0.
$
\end{definition}

For any locally integrable function $f$ on $\rn$, recall that the \emph{minimizing polynomial}
of $f$ on the cube $Q$ with degree not greater than $d$, $P_Q^d f$, is defined to be the unique
polynomial with degree not greater than $d$ such that, for any
multi-index $\alpha\in\zz^n_+$ with $|\alpha|\leq d$,
\begin{equation}\label{2}
\int_{Q}\lf[f(x)-P_Q^d f(x)\r]x^\alpha\,dx=0.
\end{equation}

\begin{definition}\label{decamp1}
Let $t\in(0,\infty)$, $q\in(0,1]$ and $\Phi$ be an Orlicz function with positive lower type $p_{\Phi}^-$ and positive upper type $p_{\Phi}^+\in(0,1]$. Let $r\in[1,\infty)$, $s\in(0,\min\{p^-_{\Phi},q\})$ and $d\in\mathbb{Z}_+$ satisfy
$d\ge\lfloor n(\frac{1}{s}-1)\rfloor$.
The \emph{Campanato space} $\mathcal{L}_{\Phi,t}^{q,r,d}(\rn)$
is defined to be the set of all measurable functions $g$
such that
$$
\|g\|_{\mathcal{L}_{\Phi,t}^{q,r,d}(\rn)}:=\sup_{B\subset\rn}\inf_{P\in\mathcal{P}_d(\rn)}
\frac{|B|}{\|\mathbf1_{B}\|_{(E_\Phi^q)_t(\rn)}}\lf[\frac{1}{|B|}\int_{B}|g(x)-P(x)|^r\,dx\r]^{\frac{1}{r}}<\infty,
$$
where the first supremum is taken over all balls $B\subset\rn$ and
$\mathcal{P}_d(\rn)$ denotes the set of all polynomials on $\rn$ with order not greater than $d$.
\end{definition}

As usual, by a little abuse of notation, we identify $f\in\mathcal{L}_{\Phi,t}^{q,r,d}(\rn)$ with an equivalent class $f+\mathcal{P}_d(\rn)$.

Borrowing some ideas from \cite[Definition 7.1]{yy}, we introduce the following notions of
local Campanato spaces.
\begin{definition}\label{decamp}
Let $t\in(0,\infty)$, $q\in(0,1]$ and $\Phi$ be an Orlicz function with positive lower type $p_{\Phi}^-$ and positive upper type $p_{\Phi}^+\in(0,1]$. Let $r\in[1,\infty)$, $s\in(0,\min\{p^-_{\Phi},q\})$ and $d\in\mathbb{Z}_+$ satisfy
$d\ge\lfloor n(\frac{1}{s}-1)\rfloor$.
The space $\mathcal{L}_{\Phi,t,\loc}^{q,r,d}(\rn)$
is defined to be the set of all measurable functions $g$
such that
\begin{align*}
\|g\|_{\mathcal{L}_{\Phi,t,\loc}^{q,r,d}(\rn)}&:=
\sup_{\mathrm{cube\,}Q\subset\rn,\ell(Q)<1}
\frac{|Q|}{\|\mathbf1_{Q}\|_{(E_\Phi^q)_t(\rn)}}\lf[\frac{1}{|Q|}\int_{Q}|g(x)-P_Q^d g(x)|^r\,dx\r]^{\frac{1}{r}}\\
&\quad+\sup_{\mathrm{cube\,}Q\subset\rn,\ell(Q)\geq1}
\frac{|Q|}{\|\mathbf1_{Q}\|_{(E_\Phi^q)_t(\rn)}}\lf[\frac{1}{|Q|}
\int_{Q}|g(x)|^r\,dx\r]^{\frac{1}{r}}
<\infty,
\end{align*}
where $P_Q^d f$ is as in \eqref{2}.
\end{definition}

In what follows, for any $r\in[1,\infty]$,
let $r'$ be its \emph{conjugate number}, namely,
$\frac1r+\frac{1}{r'}=1$.

\begin{theorem}\label{dual2}
Let $t\in(0,\infty)$, $q\in(0,1]$ and $\Phi$ be an Orlicz function with positive lower type $p_{\Phi}^-$ and positive upper type $p_{\Phi}^+\in(0,1]$. Let $r\in(1,\infty]$, $s \in(0,\min\{p^-_{\Phi},q\})$ and $d\in\mathbb{Z}_+$ satisfy $d\ge\lfloor n(\frac{1}{s}-1)\rfloor$.
Then the dual space of $(hE_\Phi^q)_t(\rn)$, denoted by $((hE_\Phi^q)_t(\rn))^*$, is $\mathcal{L}_{\Phi,t,\loc}^{q,r',d}(\rn)$
in the following sense:
\begin{enumerate}
\item[{\rm(i)}] Any
$g\in\mathcal{L}_{\Phi,t,\loc}^{q,r',d}(\rn)$ induces a linear
functional given by
\begin{equation}\label{linf}
L_g:\ f\mapsto\ L_g(f):=\int_{\rn}f(x)g(x)dx,
\end{equation}
which is initially defined on
$(hE_\Phi^{q,r,d})_t^{\fin}(\rn)$ and  has a bounded
extension to $(hE_\Phi^q)_t(\rn)$.

\item[{\rm(ii)}] Conversely, any continuous linear
functional on $(hE_\Phi^q)_t(\rn)$ is of the form \eqref{linf}
for some $g\in\mathcal{L}_{\Phi,t,\loc}^{q,r',d}(\rn)$.
\end{enumerate}
Moreover, in any case, $\|g\|_{\mathcal{L}_{\Phi,t,\loc}^{q,r',d}(\rn)}$ is
equivalent to $\|L_g\|_{((hE_\Phi^q)_t(\rn))^*}$ with the positive equivalence constants independent of $t$,
here and thereafter, $\|\cdot\|_{((hE_\Phi^q)_t(\rn))^*}$ denotes the norm of $((hE_\Phi^q)_t(\rn))^*$.
\end{theorem}

To show Theorem \ref{dual2}, we need the following technical lemma which
is just \cite[Lemma 5.4]{zyyw}.

\begin{lemma}\label{quasi}
Let $t\in(0,\infty)$, $q\in(0,1]$ and $\Phi$ be an Orlicz function with positive lower type $p_{\Phi}^-$ and positive upper type $p_{\Phi}^+\in(0,1]$.
Then there exists a nonnegative constant $C$ such that, for any sequence $\{f_j\}_{j\in\nn}\subset
(E_\Phi^{q})_t(\rn)$ of nonnegative functions
such that $\sum_{j\in\mathbb{N}}f_j$ converges
in $(E_\Phi^{q})_t(\rn),$
$$
\lf\|\sum_{j\in\mathbb{N}}f_j\r\|_{(E_\Phi^{q})_t(\rn)}
\ge C\sum_{j\in\mathbb{N}}\lf\|f_j\r\|_{(E_\Phi^{q})_t(\rn)}.
$$
\end{lemma}

\begin{proof}[Proof of Theorem \ref{dual2}]
We first show (i). To prove
$\mathcal{L}_{\Phi,t,\loc}^{q,r',d}(\rn)\subset((hE_\Phi^q)_t(\rn))^*$,
by Theorem \ref{atom ch}, it suffices
to show that
$$
\mathcal{L}_{\Phi,t,\loc}^{q,r',d}(\rn)\subset((hE_\Phi^q)_t^{r,d}(\rn))^*.
$$
Let $g\in\mathcal{L}_{\Phi,t,\loc}^{q,r',d}(\rn)$
and $a$ be a local $((E_\Phi^q)_t(\rn),\ r,\ d)$-atom
supported in a cube $Q\subset\rn$.
When $\ell(Q)\in(0,1)$, by the
moment condition of $a$, the H\"older inequality and the size condition of $a$, we know
that
\begin{align}\label{3}
|L_g(a)|:&=\lf|\int_{\rn}a(x)g(x)\,dx\r|
=\lf|\int_{\rn}a(x)\lf[g(x)-P_Q^d g(x)\r]\,dx\r|\\ \noz
&\leq\|a\|_{L^r(\rn)}\lf[\int_{Q}|g(x)-P_Q^d g(x)|^{r'}\,dx\r]^{\frac{1}{r'}}
\leq\|g\|_{\mathcal{L}_{\Phi,t,\loc}^{q,r',d}(\rn)},
\end{align}
where $P_Q^d g$ is as in \eqref{2}.
When $\ell(Q)\in[1,\infty)$, by the H\"older inequality and
the size condition of $a$, we conclude
that
\begin{align}\label{551}
|L_g(a)|:&=\lf|\int_{\rn}a(x)g(x)\,dx\r|
\leq\|a\|_{L^r(\rn)}\lf[\int_{Q}|g(x)|^{r'}\,dx\r]^{\frac{1}{r'}}
\leq\|g\|_{\mathcal{L}_{\Phi,t,\loc}^{q,r',d}(\rn)}.
\end{align}
Moreover, for any $f\in(hE_\Phi^{q,r,d})_t^{\fin}(\rn)$, by Definition \ref{definite}, we know that there exist
a sequence $\{a_j\}_{j=1}^m$ of local
$((E_\Phi^q)_t(\rn),\ r,\ d)$-atoms supported, respectively, in cubes
$\{Q_j\}_{j=1}^m$ and a sequence $\{\lambda_j\}_{j=1}^m\subset[0,\infty)$
such that, for any given $s\in (0,\min\{p^-_{\Phi},q\})$,
\begin{align}\label{4}
\lf\|\lf\{\sum_{j=1}^{m} \lf[\frac{\lambda_j}
{\|\mathbf1_{Q_j}\|_{(E_\Phi^q)_t(\rn)}} \r]^s\mathbf1_{Q_j} \r\}
^{\frac{1}{s}}\r\|_{(E_\Phi^q)_t(\rn)}
\lesssim\|f\|_{(hE_\Phi^{q,r,d})_t^{\fin}(\rn)}.
\end{align}
Then, from \eqref{3}, \eqref{551}, Lemma \ref{quasi} and \eqref{4}, it follows that
\begin{align*}
|L_g(f)|
&\leq\sum_{j=1}^{m}\lambda_j|L_g(a_j)|
\lesssim\sum_{j=1}^{m}\lambda_j\|g\|_{\mathcal{L}_{\Phi,t,\loc}^{q,r',d}(\rn)}
\lesssim\lf\|\sum_{j=1}^{m} \frac{\lambda_j}
{\|\mathbf1_{Q_j}\|_{(E_\Phi^q)_t(\rn)}} \mathbf1_{Q_j}
\r\|_{(E_\Phi^q)_t(\rn)}\|g\|_{\mathcal{L}_{\Phi,t,\loc}^{q,r',d}(\rn)}\\
&\lesssim\lf\|\lf\{\sum_{j=1}^{m} \lf[\frac{\lambda_j}
{\|\mathbf1_{Q_j}\|_{(E_\Phi^q)_t(\rn)}} \r]^s\mathbf1_{Q_j} \r\}
^{\frac{1}{s}}\r\|_{(E_\Phi^q)_t(\rn)}
\|g\|_{\mathcal{L}_{\Phi,t,\loc}^{q,r',d}(\rn)}\\
&\lesssim\|f\|_{(hE_\Phi^{q,r,d})_t^{\fin}(\rn)}
\|g\|_{\mathcal{L}_{\Phi,t,\loc}^{q,r',d}(\rn)}.
\end{align*}
By this and the fact that $(hE_\Phi^{q,r,d})_t^{\fin}(\rn)$
is dense in $(hE_\Phi^q)_t(\rn)$,
together with Theorem
\ref{finite}, we obtain the desired conclusion of (i).

As for (ii), take any cube $Q\subset\rn$ with $\ell(Q)\in[1,\infty)$.
We first prove that
\begin{equation}\label{5}
((hE_\Phi^q)_t^{r,d}(\rn))^*\subset(L^r(Q))^*,
\end{equation}
where $L^r(Q):=\{f\in L^r(\rn):\ \supp(f)\subset Q \}$.
Obviously, for any given $f\in L^r(Q)$ and $\|f\|_{L^r(Q)}\neq0$,
$a:=\frac{|Q|^{\frac{1}{r }}}{\|\mathbf1_Q\|_{(E_\Phi^q)_t(\rn)}}
\|f\|_{L^r(Q)}^{-1}f$ is a local $((E_\Phi^q)_t(\rn),r,d)$-atom.
Thus, $f\in(hE_\Phi^{q,r,d})_t(\rn)$ and
$$
\|f\|_{(hE_\Phi^{q,r,d})_t(\rn)}=\|f\|_{L^r(Q)}
\frac{\|\mathbf1_Q\|_{(E_\Phi^q)_t(\rn)}}{|Q|^{\frac{1}{r }}}
\|a\|_{(hE_\Phi^{q,r,d})_t(\rn)}\lesssim
\frac{\|\mathbf1_Q\|_{(E_\Phi^q)_t(\rn)}}{|Q|^{\frac{1}{r }}}\|f\|_{L^r(Q)},
$$
which implies that, for any $L\in((hE_\Phi^q)_t^{r,d}(\rn))^*$ and $f\in L^r(Q)$.
$$
|L(f)|\leq\|L\|_{((hE_\Phi^q)_t^{r,d}(\rn))^*}\|f\|_{(hE_\Phi^q)_t^{r,d}(\rn)}
\lesssim\|L\|_{((hE_\Phi^q)_t^{r,d}(\rn))^*}
\frac{\|\mathbf1_Q\|_{(E_\Phi^q)_t(\rn)}}{|Q|^{\frac{1}{r }}}\|f\|_{L^r(Q)}.
$$
From this, we deduce that $L\in(L^r(Q))^*$ and hence \eqref{5} holds true.

Let $\{Q_j\}_{j\in\nn}$ be a sequence of cubes satisfying that,
for any $j\in\nn$, $Q_j\subset Q_{j+1}$,
$\cup_{j\in\nn}Q_j=\rn$ and $\ell(Q_1)\in[1,\infty)$.
Assume that $L\in((hE_\Phi^q)_t^{r,d}(\rn))^*$.
Using \eqref{5}, by an argument similar to that used in the estimation of \cite[(7.16)]{yy},
we know that there
exists a function $g$ on $\rn$ such that, for any $j\in\nn$ and
$f\in L^r(Q_j)$,
\begin{equation}\label{q2}
L(f)=\int_{Q_j}f(x)g(x)\,dx.
\end{equation}

We claim that, for any $f\in(hE_\Phi^{q,r,d})_t^{\fin}(\rn)$,
\begin{equation}\label{6}
L(f)=\int_{\rn}f(x)g(x)\,dx.
\end{equation}
Indeed, from the fact that $\cup_{j\in\nn}Q_j=\rn$, it follows that,
for any local $((E_\Phi^q)_t(\rn),r,d)$-atom $a$, there exists
a $j_0\in\nn$ such that $a\in L^r(Q_{j_0})$. By this and
\eqref{q2}, we conclude that the claim \eqref{6} holds true.

We now show that $g\in\mathcal{L}_{\Phi,t,\loc}^{q,r',d}(\rn)$.
We may assume that $g\neq0$.
Observe that, for any cube $Q\subset \rn$ with $\ell(Q)\in [1,\infty)$
and any $f\in L^r(Q)$ with $\|f\|_{L^r(Q)}\leq1$,
$a:=\frac{|Q|^{\frac{1}{r }}}{\|\mathbf1_Q\|_{(E_\Phi^q)_t(\rn)}}
f$ is a local $((E_\Phi^q)_t(\rn),r,d)$-atom and $\supp(a)\subset Q$.
From this, \eqref{q2} and the fact that
$L\in((hE_\Phi^q)_t^{r,d}(\rn))^*$, we deduce that
$$
\lf|\int_{Q}a(x)g(x)\,dx\r|=|L(a)|
\leq\|L\|_{((hE_\Phi^q)_t^{r,d}(\rn))^*}\|a\|_{(hE_\Phi^q)_t(\rn)}
\lesssim\|L\|_{((hE_\Phi^q)_t^{r,d}(\rn))^*},
$$
which implies that, for any $f\in L^r(Q)$ with $\|f\|_{L^r(Q)}\leq1$,
\begin{equation}\label{q5}
\frac{|Q|^{\frac{1}{r }}}{\|\mathbf1_Q\|_{(E_\Phi^q)_t(\rn)}}
\lf|\int_{Q}f(x)g(x)\,dx\r|\lesssim\|L\|_{((hE_\Phi^q)_t^{r,d}(\rn))^*}.
\end{equation}
When $r=\infty$, let $f:=\sign(g)$ and, when $r\in(1,\infty)$, let $f:=\frac{|g|^{r'-1}}{\||g|^{r'-1}\|_{L^r(\rn)}}$.
By this, the H\"older inequality and \eqref{q5}, we conclude that
\begin{equation}\label{7}
\frac{|Q|}{\|\mathbf1_{Q}\|_{(E_\Phi^q)_t(\rn)}}\lf[\frac{1}{|Q|}
\int_{Q}|g(x)|^{r'}\,dx\r]^{\frac{1}{r'}}\lesssim\|L\|_{((hE_\Phi^q)_t^{r,d}(\rn))^*}.
\end{equation}
Moreover, from $(HE_\Phi^q)_t(\rn)\subset(hE_\Phi^q)_t(\rn)$ and the fact that,
for any $f\in(HE_\Phi^q)_t(\rn)$,
$$\|f\|_{(hE_\Phi^q)_t(\rn)}\leq\|f\|_{(HE_\Phi^q)_t(\rn)},$$
we deduce that $((hE_\Phi^q)_t(\rn))^*\subset((HE_\Phi^q)_t(\rn))^*$ and
$L|_{(HE_\Phi^q)_t(\rn)}\in((HE_\Phi^q)_t(\rn))^*$.
Since \eqref{6} holds true for any $f\in(hE_\Phi^{q,r,d})_t^{\fin}(\rn)$,
from \cite[Theorem 5.7]{zyyw}, we deduce that
$g\in\mathcal{L}_{\Phi,t}^{q,r',d}(\rn)$ and
$$
\|g\|_{\mathcal{L}_{\Phi,t}^{q,r',d}(\rn)}\lesssim
\lf\|L|_{(HE_\Phi^q)_t(\rn)}\r\|_{((HE_\Phi^q)_t(\rn))^*}\lesssim
\|L\|_{((hE_\Phi^{q,r,d})_t(\rn))^*}.
$$
This, combined with
\eqref{7} and the definition of $\mathcal{L}_{\Phi,t,\loc}^{q,r',d}(\rn)$,
implies that $g\in\mathcal{L}_{\Phi,t,\loc}^{q,r',d}(\rn)$ and
$$
\|g\|_{\mathcal{L}_{\Phi,t,\loc}^{q,r',d}(\rn)}\lesssim
\|L\|_{((hE_\Phi^{q,r,d})_t(\rn))^*},
$$
which completes the proof of (ii) and hence of Theorem \ref{dual2}.
\end{proof}

\begin{remark}
Let $t=1$, $q\in(0,\infty)$ and $\Phi(\tau):=\tau^q$ for any $\tau\in[0,\infty)$.
Then $(hE_\Phi^q)_t(\rn)$ and $(E_\Phi^q)_t(\rn)$ become, respectively, the classical
local Hardy space $h^q(\rn)$ and Lebesgue space $L^q(\rn)$.
When $q=1$, we have $(hE_\Phi^q)_t(\rn)=h^1(\rn)$ and, moreover,
$\mathcal{L}_{\Phi,t,\loc}^{q,r,0}(\rn)$ with
$r\in[1,\infty)$ coincides with the classical local {\rm BMO} space
$\bmo(\rn)$ which was introduced by Goldberg in \cite{g}.
\end{remark}

\section{Sharpness of bilinear decompositions}\label{s5}

In this section, we
prove that the
bilinear decomposition
\eqref{cao1}
is sharp and the bilinear decomposition
\eqref{cao2} is not sharp
(see Remark \ref{mm} below).
These two results of bilinear decomposition were proved in \cite{cky1}.

Now, we recall the notions of both the local Hardy-type space $h_*^\Phi(\rn)$
introduced in \cite{bf} and the local Musielak--Orlicz Hardy
space $h^{\log}(\rn)$ in \cite{yy}.

\begin{definition}\label{cjs}
For any $\tau\in[0,\infty)$, let
\begin{equation}\label{41}
\Phi(\tau):=\frac{\tau}{\log(e+\tau)}.
\end{equation}
\begin{itemize}
\item[(i)] The \emph{variant Orlicz space} $L_*^\Phi(\rn)$ is defined to be the set of all measurable functions $f$
such that $$\|f\|_{L_*^\Phi(\rn)}
:=\sum_{k\in\zz^n}\|f\mathbf1_{Q_k}\|_{L^{\Phi}(\rn)}<\infty,$$
where $Q_{k}:=k+[0,1)^n$ for any $k\in\mathbb{Z}^n$.
\item[(ii)] The \emph{variant local Orlicz Hardy space} $h_*^\Phi(\rn)$ is defined
by setting
 $$h_*^\Phi(\rn):=\lf\{f\in\mathcal{S}'(\rn):\ \|f\|_{h_*^\Phi(\rn)}:=
\|m_b^{**}(f,\varphi)\|_{L_*^\Phi(\rn)}<\infty \r\},$$
where $\varphi\in\mathcal{S}(\rn)$ satisfies $\int_{\rn}\varphi(x)\,dx\neq0$ and $m_b^{**}(f,\varphi)$ is as in (\ref{51}) with $b\in(2n,\infty)$.
\item[(iii)] The \emph{local Orlicz Hardy space}
$h^\Phi(\rn)$ is defined via replacing $L_*^\Phi(\rn)$ in (ii) by $L^\Phi(\rn)$.
\item[(iv)] The \emph{local Hardy space}
$h^1(\rn)$ is defined via replacing $L_*^\Phi(\rn)$ in (ii) by $L^1(\rn)$.
\end{itemize}
\end{definition}

\begin{definition}\label{mshs}
Let $\theta$ be as in \eqref{2221}.
\begin{itemize}
\item[(i)] The \emph{Musielak--Orlicz space} $L^\theta(\rn)$
is defined to be the set of all measurable functions $f$
such that $$\|f\|_{L^\theta(\rn)}
:=\inf\lf\{\lambda\in(0,\infty):\ \int_{\rn}\theta
\lf(x,\,\frac{|f(x)|}{\lambda}\r)\,dx\leq1 \r\}<\infty.$$
\item[(ii)] The \emph{local Musielak--Orlicz Hardy
space} $h^{\log}(\rn)$ is defined
by setting
\begin{align*}
h^{\log}(\rn):
&=\lf\{
f\in\mathcal{S}'(\rn):\ \|m_b^{**}(f,\varphi)\|_{L^{\theta}(\rn)}<\infty
\r\},
\end{align*}
where
$\varphi\in\mathcal{S}(\rn)$ satisfies$\int_{\rn}\varphi(x)\,dx\neq0$
and $m_b^{**}(f,\varphi)$ is as in (\ref{51}) with $b\in(2n,\infty)$.
\end{itemize}
\end{definition}

To prove that the bilinear decomposition
\eqref{cao2} is not sharp, we need the following very useful
technical lemma.

\begin{lemma}\label{po}
Let $\theta$ be as in \eqref{2221} and $\Phi$ be as in \eqref{41}.
Then $L_*^\Phi(\rn)\subset L^\theta(\rn)$ and there exists a positive constant $C$ such that,
for any $f\in L_*^\Phi(\rn)$,
$$
\lf\|f\r\|_{L^\theta(\rn)}\le C\lf\|f\r\|_{L_*^\Phi(\rn)}.
$$
\end{lemma}

\begin{proof}
Let $f\in L_*^\Phi(\rn)$.
By the homogeneity of both $\|\cdot\|_{L_*^\Phi(\rn)}$ and $\|\cdot\|_{L^\theta(\rn)}$,
without loss of generality,
we may assume that $\|f\|_{L_*^\Phi(\rn)}=1$.
To show this lemma, by \cite[Lemma 1.1.11(i)]{ylk}, we know that it suffices to prove that
\begin{align}\label{eqz}
\int_{\rn}\theta
\lf(x,\,|f(x)|\r)\,dx\lesssim1.
\end{align}
We first write
\begin{align}\label{eq611}
\int_{\rn}\theta
\lf(x,\,|f(x)|\r)\,dx&=\sum_{k\in\zz^n}\int_{Q_k}\theta
\lf(x,\,|f(x)|\r)\,dx\\\noz
&=
\sum_{k\in\zz^n,|k|\le2\sqrt{n}}\int_{Q_k}\theta
\lf(x,\,|f(x)|\r)\,dx
+\sum_{k\in\zz^n,|k|>2\sqrt{n}}\cdots
=:\mathrm{I}_1+\mathrm{I}_2,
\end{align}
where $Q_{k}:=k+[0,1)^n$ for any $k\in\mathbb{Z}^n$.

We first deal with $\mathrm{I}_1$. By Definition \ref{d1.2}, we conclude that
\begin{align}\label{eq6}
\mathrm{I}_1&=
\sum_{k\in\zz^n,|k|\le2\sqrt{n}}\int_{Q_k}\theta
\lf(x,\,|f(x)|\r)\,dx
=\sum_{k\in\zz^n,|k|\le2\sqrt{n}}\int_{Q_k}\frac{|f(x)|}{\log(e+|x|)+\log(e+|f(x)|)}\,dx\\\noz
&=\sum_{k\in\zz^n,|k|\le2\sqrt{n}}[\|f\mathbf1_{Q_k}\|_{L^\Phi(\rn)}+1]
\int_{Q_k}\frac{|f(x)|/[\|f\mathbf1_{Q_k}\|_{L^\Phi(\rn)}+1]}{\log(e+|f(x)|/[\|f\mathbf1_{Q_k}\|_{L^\Phi(\rn)}+1])}\\\noz
&\quad\times\frac{\log(e+|f(x)|/[\|f\mathbf1_{Q_k}\|_{L^\Phi(\rn)}+1])}{\log(e+|x|)+\log(e+|f(x)|)}\,dx\\\noz
&\le\sum_{k\in\zz^n,|k|\le 2\sqrt{n}}[\|f\mathbf1_{Q_k}\|_{L^\Phi(\rn)}+1]
\int_{Q_k}\frac{|f(x)|/[\|f\mathbf1_{Q_k}\|_{L^\Phi(\rn)}+1]}{\log(e+|f(x)|/[\|f\mathbf1_{Q_k}\|_{L^\Phi(\rn)}+1])}\,dx\\\noz
&=\sum_{k\in\zz^n,|k|\le 2\sqrt{n}}[\|f\mathbf1_{Q_k}\|_{L^\Phi(\rn)}+1]
\int_{Q_k}\Phi\lf(\frac{|f(x)|}{\|f\mathbf1_{Q_k}\|
_{L^\Phi(\rn)}+1}\r)\,dx\\\noz
&\lesssim1+\sum_{k\in\zz^n}\|f\mathbf1_{Q_k}\|_{L^\Phi(\rn)}\lesssim1.
\end{align}
As for $\mathrm{I_{2}}$, observe that, for any $k\in\zz^n,|k|>2\sqrt{n}$ and $x\in Q_k$,
$
|x|>|k|-\sqrt{n}>\frac{1}{2}|k|.
$
From this and Definition \ref{d1.2}, if follows that
\begin{align*}
\mathrm{I}_2&
=(n+1)\sum_{k\in\zz^n,|k|>2\sqrt{n}}\int_{Q_k}\frac{|f(x)|}{(n+1)[\log(e+|x|)+\log(e+|f(x)|)]}\,dx\\\noz
&\sim\sum_{k\in\zz^n,|k|>2\sqrt{n}}\lf[\|f\mathbf1_{Q_k}\|_{L^\Phi(\rn)}+|k/2|^{-(n+1)}\r]
\int_{Q_k}\frac{|f(x)|/[\|f\mathbf1_{Q_k}\|_{L^\Phi(\rn)}+|k/2|^{-(n+1)}]}
{\log(e+|f(x)|/[\|f\mathbf1_{Q_k}\|_{L^\Phi(\rn)}+|k/2|^{-(n+1)}])}\\\noz
&\quad\times\frac{\log(e+|f(x)|/[\|f\mathbf1_{Q_k}\|_{L^\Phi(\rn)}+|k/2|^{-(n+1)}])}{\log(e+|x|^{n+1})+\log(e+|f(x)|)}\,dx\\\noz
&\lesssim\sum_{k\in\zz^n,|k|>2\sqrt{n}}\lf[\|f\mathbf1_{Q_k}\|_{L^\Phi(\rn)}+|k/2|^{-(n+1)}\r]
\int_{Q_k}\frac{|f(x)|/[\|f\mathbf1_{Q_k}\|_{L^\Phi(\rn)}+|k/2|^{-(n+1)}]}
{\log(e+|f(x)|/[\|f\mathbf1_{Q_k}\|_{L^\Phi(\rn)}+|k/2|^{-(n+1)}])}\\\noz
&\quad\times\frac{\log(e+|k/2|^{n+1}|f(x)|)}{\log(e+|x|^{n+1}|f(x)|)}\,dx\\\noz
&\lesssim\sum_{k\in\zz^n,|k|>2\sqrt{n}}
\lf[\|f\mathbf1_{Q_k}\|_{L^\Phi(\rn)}+|k/2|^{-(n+1)}\r]
\int_{Q_k}\Phi\lf(\frac{|f(x)|}{\|f\mathbf1_{Q_k}\|
_{L^\Phi(\rn)}+|k/2|^{-(n+1)}}\r)\,dx\\\noz
&\le\sum_{k\in\zz^n,|k|>2\sqrt{n}}
\lf[\|f\mathbf1_{Q_k}\|_{L^\Phi(\rn)}+|k/2|^{-(n+1)}\r]
\lesssim1+\sum_{k\in\zz^n,|k|>2\sqrt{n}}|k|^{-(n+1)}
\lesssim1,
\end{align*}
which, combined with \eqref{eq611} and \eqref{eq6}, further implies that \eqref{eqz} holds true.
This finishes the proof of Lemma \ref{po}.
\end{proof}

\begin{lemma}\label{888}
Let $\Phi$ be as in \eqref{41}. Then, for any ball
$B\subset\rn$,
$$
\lf\|\mathbf1_B\r\|_{L_*^\Phi(\rn)}\sim\frac{|B|}{\log(e+\frac{1}{|B|})}\sim\|\mathbf1_B\|_{L^\Phi(\rn)},
$$
where the positive equivalence constants are independent of $B$.
\end{lemma}

\begin{proof}
For any ball $B\subset\rn$, we denote by $x_B$ its center and by $r_B$ its radius.
Recall that the following equivalence
\begin{equation}\label{222}
\|\mathbf1_B\|_{L^\Phi(\rn)}\sim\frac{|B|}{\log(e+\frac{1}{|B|})}
\end{equation}
was established in \cite[Lemma 7.13]{yy2}.
To estimate $\|\mathbf1_B\|_{L_*^\Phi(\rn)}$,
we consider the following two cases.

If $r_B\in[\frac{1}{2},\infty)$, then, by Remark \ref{ReHpq}(ii) and \eqref{222}, we have
\begin{align*}
\|\mathbf1_B\|_{L_*^\Phi(\rn)}&\sim\|\mathbf1_B\|_{(E_\Phi^1)_1(\rn)}
\sim\int_{\rn}\|\mathbf1_{B(y,1)}\mathbf1_{B}\|_{L^\Phi(\rn)}\,dy
\sim\int_{B(x_B,3r_B)}\|\mathbf1_{B(y,1)}\mathbf1_{B}\|_{L^\Phi(\rn)}\,dy\\\noz
&\lesssim|B(x_B,3r_B)|
\sim\frac{|B|}{\log(e+\frac{1}{|B|})}
\end{align*}
and
\begin{align*}
\lf\|\mathbf1_B\r\|_{L_*^\Phi(\rn)}&\sim\lf\|\mathbf1_B\r\|_{(E_\Phi^1)_1(\rn)}
\sim\int_{\rn}\lf\|\mathbf1_{B(y,1)}\mathbf1_{B}\r\|_{L^\Phi(\rn)}\,dy
\gtrsim\int_{B(x_B,\frac12r_B)}\lf\|\mathbf1_{B(y,\frac14)}\r\|_{L^\Phi(\rn)}\,dy\\\noz
&\sim\lf|B\lf(x_B,\frac12r_B\r)\r|
\sim\frac{|B|}{\log(e+\frac{1}{|B|})}.
\end{align*}
Thus, $\lf\|\mathbf1_B\r\|_{L_*^\Phi(\rn)}\sim\frac{|B|}{\log(e+\frac{1}{|B|})}$
when $r_B\in[\frac{1}{2},\infty)$.

If $r_B\in(0,\frac{1}{2})$, then, by \eqref{222}, we obtain
$$
\|\mathbf1_B\|_{L_*^\Phi(\rn)}\sim\int_{\rn}\|\mathbf1_{B(y,1)}\mathbf1_{B}\|_{L^\Phi(\rn)}\,dy
\lesssim\int_{B(x_B,2)}\|\mathbf1_{B}\|_{L^\Phi(\rn)}\,dy
\sim\frac{|B|}{\log(e+\frac{1}{|B|})}
$$
and
$$
\|\mathbf1_B\|_{L_*^\Phi(\rn)}\sim\int_{\rn}\|\mathbf1_{B(y,1)}\mathbf1_{B}\|_{L^\Phi(\rn)}\,dy
\gtrsim\int_{B(x_B,\frac12)}\|\mathbf1_{B}\|_{L^\Phi(\rn)}\,dy
\sim\frac{|B|}{\log(e+\frac{1}{|B|})}.
$$
Thus, $\lf\|\mathbf1_B\r\|_{L_*^\Phi(\rn)}\sim\frac{|B|}{\log(e+\frac{1}{|B|})}$
also holds true when $r_B\in(0,\frac{1}{2})$, which completes the proof of Lemma \ref{888}.
\end{proof}

We now recall some notions about
local BMO spaces in \cite{yy2}.

\begin{definition}\label{bmo1}
\begin{itemize}
\item[(i)] The \emph{local {\rm BMO} space} $\bmo(\rn)$
is defined to be the set of all measurable functions $f$ such that
\begin{equation*}
\|f\|_{\bmo(\rn)}:=\sup_{\mathrm{cube\,}Q\subset\rn,\ell(Q)<1}
\frac{1}{|Q|}\int_{Q}|f(x)-f_Q|\,dx
+\sup_{\mathrm{cube\,}Q\subset\rn,\ell(Q)\geq1}\frac{1}{|Q|}\int_{Q}|f(x)|\,dx<\infty,
\end{equation*}
where $f_Q:=\frac{1}{|Q|}\int_{Q}f(x)\,dx$ for any cube $Q\subset\rn$.
\item[(ii)]
Let $\Phi$ be as in \eqref{41}.
The \emph{local {\rm BMO}-type space ${\bmo}^{\Phi}(\rn)$} is defined
to be the set of all measurable functions $f\in L_{\loc}^1(\rn)$ such that
\begin{align*}
\|f\|_{{\bmo}^{\Phi}(\rn)}:&=
\sup_{\mathrm{cube\,}Q\subset\rn,\ell(Q)<1}
\frac{\log(e+\frac{1}{|Q|})}{|Q|}\int_{Q}|f(x)-f_Q|\,dx\\\noz
&\quad+\sup_{\mathrm{cube\,}Q\subset\rn,\ell(Q)\geq1}
\frac{\log(e+\frac{1}{|Q|})}{|Q|}
\int_{Q}|f(x)|\,dx
<\infty,
\end{align*}
where $f_Q:=\frac{1}{|Q|}\int_{Q}f(x)\,dx$ for any cube $Q\subset\rn$.
\item[(iii)]
Let $\theta$ be as in \eqref{2221}.
The \emph{local {\rm BMO}-type space} $\bmo^{\log}(\rn)$ is defined
to be the set of all measurable functions $f\in L_{\loc}^1(\rn)$ such that
\begin{align*}
\|f\|_{\bmo^{\log}(\rn)}:&=
\sup_{\mathrm{cube\,}Q\subset\rn,\ell(Q)<1}
\frac{\log(e+\frac{1}{|Q|})+\sup_{x\in Q}[\log(e+|x|)]}{|Q|}\int_{Q}|f(x)-f_Q|\,dx\\
&\quad+\sup_{\mathrm{cube\,}Q\subset\rn,\ell(Q)\geq1}
\frac{\log(e+\frac{1}{|Q|})+\sup_{x\in Q}[\log(e+|x|)]}{|Q|}
\int_{Q}|f(x)|\,dx
<\infty,
\end{align*}
where $f_Q:=\frac{1}{|Q|}\int_{Q}f(x)\,dx$ for any cube $Q\subset\rn$.
\end{itemize}
\end{definition}

Recall that a function $g$ on $\rn$ is called a
\emph{pointwise multiplier on} ${\bmo}^{\Phi}(\rn)$ if the pointwise
multiplication $fg$ belongs to ${\bmo}^{\Phi}(\rn)$ for any $f\in{\bmo}^{\Phi}(\rn)$.

\begin{theorem}\label{cheng}
Let $\Phi$ be as in \eqref{41}.
\begin{itemize}
\item[\rm(i)]
The dual space of $h_*^\Phi(\rn)$, denoted by $(h_*^\Phi(\rn))^*$,
is ${\bmo}^{\Phi}(\rn)$.
\item[\rm(ii)]
The dual space of $h^{\log}(\rn)$, denoted by $(h^{\log}(\rn))^*$,
is $\bmo^{\log}(\rn)$.
\item[\rm(iii)] The class of pointwise multipliers of $\bmo(\rn)$ is
$L^\infty(\rn)\cap (h_*^\Phi(\rn))^*$.
\item[\rm(iv)] $\bmo^{\log}(\rn)\subsetneqq{\bmo}^{\Phi}(\rn)$
and $L^\infty(\rn)\cap\bmo^{\log}(\rn)\subsetneqq L^\infty(\rn)\cap{\bmo}^{\Phi}(\rn)$.
\item[\rm(v)] $h_*^\Phi(\rn)\subsetneqq h^{\log}(\rn)$ and hence $L_*^\Phi(\rn)\subsetneqq L^\theta(\rn)$.
\end{itemize}
\end{theorem}

\begin{proof}
From Theorem \ref{dual2}, Lemma \ref{888} and Definition \ref{bmo1}, it is easy to deduce that
(i) holds true.

By the proof of \cite[Proposition 7.2]{Ky14} and \cite[Corollary 7.6]{yy}, we easily obtain (ii).

As for (iii), by \cite[Lemma 7.13 and Corollary 7.6]{yy2} and (i), we know that the dual space of
$h^\Phi(\rn)$ in \cite[Theorem 7.9(ii)]{yy2} coincides with the dual space of $h_*^\Phi(\rn)$.
From the proof of \cite[Theorem 7.9(ii)]{yy2},
it follows that the class of pointwise multipliers of $\bmo(\rn)$ is the space
$L^\infty(\rn)\cap(h^\Phi(\rn))^*$.
Thus,
the class of pointwise
multipliers for $\bmo(\rn)$ is the space
$L^\infty(\rn)\cap (h_*^\Phi(\rn))^*$. This finishes the proof of (iii).

Now we prove (iv). Applying the definitions of $\bmo^{\log}(\rn)$ and ${\bmo}^{\Phi}(\rn)$,
it is easy to see that
$\bmo^{\log}(\rn)\subset{\bmo}^{\Phi}(\rn)$,
$1\in{\bmo}^{\Phi}(\rn)$
and $\|1\|_{{\bmo}^{\Phi}(\rn)}=\log(1+e)$, $1\notin\bmo^{\log}(\rn)$
and hence $\bmo^{\log}(\rn)\subsetneqq{\bmo}^{\Phi}(\rn)$, which,
together with $1\in L^\infty(\rn)$, further implies that
$L^\infty(\rn)\cap\bmo^{\log}(\rn)
\subsetneqq L^\infty(\rn)\cap{\bmo}^{\Phi}(\rn)$.

As for (v), from Lemma \ref{po} and the definitions of $h_*^\Phi(\rn)$ and
$h^{\log}(\rn)$, it is easy to deduce that
\begin{equation}\label{jkl3}
h_*^\Phi(\rn)\subset h^{\log}(\rn)
\end{equation} and
\begin{equation}\label{jkl2}
\|\cdot\|_{h^{\log}(\rn)}\lesssim\|\cdot\|_{h_*^\Phi(\rn)}.
\end{equation}
Now we show that $h_*^\Phi(\rn)\subsetneqq h^{\log}(\rn)$.
Assume that, as sets,
\begin{equation}\label{jkl}
h_*^\Phi(\rn)= h^{\log}(\rn).
\end{equation}
From \cite[Example 1.1.5]{ylk}, Definitions \ref{cjs}(ii) and \ref{mshs}(ii),
we deduce that $(h_*^\Phi(\rn),\|\cdot\|_{h_*^\Phi(\rn)}^{\frac12})$
and $(h^{\log}(\rn),\|\cdot\|_{h^{\log}(\rn)}^{\frac12})$ are Fr\'echet
spaces (see, for instance, \cite[p.\,52, Definition 1]{y}),
which, together with \eqref{jkl2} and the norm-equivalence theorem
(see, for instance, \cite[Corollary 2.12(d)]{R}), further implies that
$\|\cdot\|_{h_*^\Phi(\rn)}^{\frac12}\lesssim\|\cdot\|_{h^{\log}(\rn)}^{\frac12}$
and hence $\|\cdot\|_{h_*^\Phi(\rn)}\lesssim\|\cdot\|_{h^{\log}(\rn)}$.
By this, we know that,
for any $L\in (h_*^\Phi(\rn))^\ast$ and $f\in h^{\log}(\rn)$,
$$
|L(f)|\lesssim\|f\|_{h_*^\Phi(\rn)}\lesssim\|f\|_{h^{\log}(\rn)}
$$
and hence $(h_*^\Phi(\rn))^\ast\subset(h^{\log}(\rn))^\ast$,
which, combined with (i) and (ii), leads to a
contradiction with (iv). Thus, the assumption \eqref{jkl} is not true. From this and
\eqref{jkl3}, we further deduce that $h_*^\Phi(\rn)\subsetneqq h^{\log}(\rn)$.
By this and Lemma \ref{po} [namely, $L_*^\Phi(\rn)\subset L^\theta(\rn)$],
we finally conclude that $L_*^\Phi(\rn)\subsetneqq L^\theta(\rn)$.
This finishes the proof of (v) and hence of
Theorem \ref{cheng}.
\end{proof}

Recall that, for any  $f\in h^1(\rn)$ and
$g\in \bmo(\rn)$,
the {\it product} $f\times g$ is defined to be  a Schwartz
distribution in $\cs'(\rn)$ such that, for any $\phi\in \mathcal{S}(\rn)$,
\begin{align}\label{def-product}
\langle f\times g,\,\phi \rangle:=\langle \phi g,\,f \rangle,
\end{align}
where the last bracket denotes the dual pair between
$\bmo(\rn)$ and $h^1(\rn)$. From
\cite[Theorem 3]{NY85}, we deduce that every
$\phi \in \mathcal{S}(\rn)$ is a pointwise multiplier on $\bmo(\rn)$,
which implies that
equality \eqref{def-product} is well defined. By Theorem \ref{cheng}(iii), we know that
$L^\infty(\rn)\cap {\bmo}^{\Phi}(\rn)$
characterizes the class of pointwise multipliers of
$\bmo(\rn)$.
From this, it follows that
the \emph{largest range} of $\varphi$ which makes
$\langle  f\times g,\,\varphi\rangle=\langle  g\varphi, f\rangle$ meaningful is
$\varphi\in L^\infty(\rn)\cap {\bmo}^{\Phi}(\rn)$.
By Theorem \ref{cheng}(iv), we know that
$$L^\infty(\rn)\cap{\bmo}^{\log}(\rn)\subsetneqq L^\infty(\rn)\cap {\bmo}^{\Phi}(\rn).$$
Thus, $$\langle  f\times g,\,\varphi\rangle=\langle  g\varphi, f\rangle$$ is meaningful for any
$\varphi\in L^\infty(\rn)\cap \bmo^{\log}(\rn)$.

From (i) and (iii) of Theorem \ref{cheng}, similarly to \cite[Theorem 1.1(ii)]{cky1}
and \cite[Remark 3.1]{cky1},
we deduce the following conclusion (see (i) and (ii) of Lemma \ref{mainthm3} below), whose
proof is a slight modification of the one of \cite[Theorem 1.1(ii)]{cky1}, and
we omit the details.

\begin{lemma}\label{mainthm3}
\begin{itemize}
\item[\rm(i)]
Let $\Phi$ be as in \eqref{41}
and $h_*^\Phi(\rn)$ as in Definition \ref{cjs}.
Then there exist two bounded bilinear operators
$$S:\, h^1(\rn)\times \bmo(\rn)\to L^1(\rn)$$
and
$$T:\, h^1(\rn)\times \bmo(\rn)\to h_*^\Phi(\rn)$$
such that, for any $(f,g)\in h^1(\rn)\times \bmo(\rn)$,
\begin{align*}
f\times g=S(f,\,g)+T(f,\,g)\qquad in\;\cs'(\rn).
\end{align*}
Moreover, for any given $(f,g)\in h^1(\rn)\times \bmo(\rn)$ and any
$\varphi\in L^\infty(\rn)\cap {\bmo}^{\Phi}(\rn)$,
$$
\la f\times g,\varphi\ra=\la \varphi, S(f,g)\ra+\la \varphi,T(f,g)\ra.$$

\item[\rm(ii)]
Let $h^{\log}(\rn)$ be as in Definition \ref{mshs}.
Then there exist two bounded bilinear operators
$$S:\, h^1(\rn)\times \bmo(\rn)\to L^1(\rn)$$
and
$$T:\, h^1(\rn)\times \bmo(\rn)\to h^{\log}(\rn)$$
such that, for any $(f,g)\in h^1(\rn)\times \bmo(\rn)$,
\begin{align*}
f\times g=S(f,\,g)+T(f,\,g)\qquad in\;\cs'(\rn).
\end{align*}
Moreover, for any given $(f,g)\in h^1(\rn)\times \bmo(\rn)$ and any
$\varphi\in L^\infty(\rn)\cap {\bmo}^{\log}(\rn)$,
$$
\la f\times g,\varphi\ra=\la \varphi, S(f,g)\ra+\la \varphi,T(f,g)\ra.$$
\end{itemize}
\end{lemma}

We point out that the conclusions of (i) and (ii) of
Lemma \ref{mainthm3} are simply
denoted, respectively, by \eqref{cao1} and \eqref{cao2}.

\begin{theorem}\label{449}
Let $\Phi$ be as in \eqref{41}
and $h_*^\Phi(\rn)$ as in Definition \ref{cjs}. Assume that $\cy\subset h_*^\Phi(\rn)$
is a quasi-Banach spaces satisfying
$\|\cdot\|_{h_*^\Phi(\rn)}\le C\|\cdot\|_{\cy}$ with $C$ being a positive constant,
and Lemma \ref{mainthm3}(i) with $h_*^\Phi(\rn)$ therein replaced
by $\cy$. Then
$L^\infty(\rn)\cap (\cy)^\ast=L^\infty(\rn)\cap (h_*^\Phi(\rn))^\ast$.
\end{theorem}
\begin{proof}
From the fact that $\cy$ satisfies Lemma \ref{mainthm3}(i) with $h_*^\Phi(\rn)$ therein replaced
by $\cy$, it follows that, for any given $(f,g)\in h^1(\rn)\times \bmo(\rn)$,
\begin{align*}
\langle\varphi,S(f,\,g)\rangle+\langle\varphi,T(f,\,g)\rangle=
\langle f\times g,\,\varphi\rangle=\la\varphi g,f\ra
,\qquad \forall\, \varphi\in L^\infty(\rn)\cap (\cy)^\ast,
\end{align*}
where $S:\ h^1(\rn)\times\bmo(\rn)\to L^1(\rn)$ and
$T:\ h^1(\rn)\times\bmo(\rn)\to \cy$ are two bounded bilinear operators.
From this, we deduce that $\varphi$ is a pointwise multiplier of $\bmo(\rn)$, which,
by Theorem \ref{cheng}(iii), implies that $\varphi\in L^\infty(\rn)\cap (h_*^\Phi(\rn))^\ast$.
We therefore obtain
\begin{equation}\label{pp}
L^\infty(\rn)\cap (\cy)^\ast\subset L^\infty(\rn)\cap (h_*^\Phi(\rn))^\ast.
\end{equation}
From $\cy\subset h_*^\Phi(\rn)$ and the assumption that $\|\cdot\|_{h_*^\Phi(\rn)}\lesssim\|\cdot\|_{\cy}$,
we deduce that, for any $L\in (h_*^\Phi(\rn))^\ast$ and $f\in\cy$,
$$
|L(f)|\lesssim\|f\|_{h_*^\Phi(\rn)}\lesssim\|f\|_{\cy},
$$
which implies that $L\in(\cy)^\ast$ and hence $(h_*^\Phi(\rn))^\ast\subset (\cy)^\ast$.
By this and \eqref{pp},
we further conclude that
$$ L^\infty(\rn)\cap (\cy)^\ast\subset L^\infty(\rn)\cap (h_*^\Phi(\rn))^\ast\subset L^\infty(\rn)\cap (\cy)^\ast,$$
which implies that $L^\infty(\rn)\cap (\cy)^\ast=L^\infty(\rn)\cap (h_*^\Phi(\rn))^\ast$.
This finishes the proof of Theorem \ref{449}.
\end{proof}

\begin{remark}\label{mm}
\begin{itemize}
\item[\rm (i)]The sharpness of \eqref{cao1} is implied by Theorem \ref{449}. Indeed,
suppose that Lemma \ref{mainthm3}(i) holds true with $h_*^\Phi(\rn)$ therein replaced by a smaller quasi-Banach spaces $\cy$. From Theorem \ref{449}, we deduce that
$L^\infty(\rn)\cap (\cy)^\ast=L^\infty(\rn)\cap (h_*^\Phi(\rn))^\ast$.
In this sense, \eqref{cao1} [namely, Lemma \ref{mainthm3}(i)] is sharp.
It is still unknown whether or not $h_*^\Phi(\rn)$
is indeed the smallest space, in the sense of the inclusion of sets, having the property
as in Lemma \ref{mainthm3}(i); see \cite{BK}.

\item[\rm (ii)]
From Theorem \ref{cheng}(v), it follows that $h_*^\Phi(\rn)\subsetneqq h^{\log}(\rn)$.
Thus, the bilinear decomposition
in Lemma \ref{mainthm3}(ii) [namely, \eqref{cao2}]
is not sharp.
\end{itemize}
\end{remark}

\bigskip

\noindent Yangyang Zhang, Dachun Yang (Corresponding author) and Wen Yuan

\smallskip

\noindent  Laboratory of Mathematics and Complex Systems
(Ministry of Education of China),
School of Mathematical Sciences, Beijing Normal University,
Beijing 100875, People's Republic of China

\smallskip

\noindent {\it E-mails}: \texttt{yangyzhang@mail.bnu.edu.cn} (Y. Zhang)

\noindent\phantom{{\it E-mails:}} \texttt{dcyang@bnu.edu.cn} (D. Yang)

\noindent\phantom{{\it E-mails:}} \texttt{wenyuan@bnu.edu.cn} (W. Yuan)
\end{document}